\documentclass[final,onefignum,onetabnum]{siamart171218}
\usepackage{lipsum}
\usepackage{amsfonts,amssymb}
\usepackage{mathrsfs}
\usepackage{graphicx}
\usepackage{epstopdf}
\usepackage{mathtools}
\usepackage{enumerate}
\usepackage{pifont}
\usepackage{tikz}
\usepackage{tikz-cd,mathtools}
\usepackage{pdfpages}
\usepackage{algorithm}
\usepackage{algpseudocode}
\pdfoutput=1
\ifpdf
  \DeclareGraphicsExtensions{.eps,.pdf,.png,.jpg}
\else
  \DeclareGraphicsExtensions{.eps}
\fi

\newsiamremark{remark}{Remark}
\newsiamremark{hypothesis}{Hypothesis}
\crefname{hypothesis}{Hypothesis}{Hypotheses}
\newsiamthm{claim}{Claim}
\newsiamthm{example}{Example}

\let\O\undefined
\let\I\undefined
\let\T\undefined

\DeclareMathOperator{\O}{O}
\DeclareMathOperator{\I}{I}
\DeclareMathOperator{\T}{T}
\DeclareMathOperator{\tr}{tr}

\definecolor{DukeBlue}{HTML}{001A57}
\definecolor{DarkRed}{rgb}{0.75, 0.0, 0.0}
\definecolor{DarkGreen}{rgb}{0.0, 0.5, 0.0}

\headers{Semi-Riemannian Manifold Optimization}{T. Gao, L.-H. Lim, and K. Ye}

\title{Semi-Riemannian Manifold Optimization}%\thanks{Submitted to the editors DATE.
\author{Tingran Gao\thanks{Department of Statistics and Committee on Computational and Applied Mathematics (CCAM), The University of Chicago, Chicago IL (\email{tingrangao@galton.uchicago.edu})}
\and Lek-Heng Lim \thanks{Department of Statistics and Committee on Computational and Applied Mathematics (CCAM), The University of Chicago, Chicago, IL (\email{lekheng@galton.uchicago.edu})}
\and Ke Ye \thanks{Academy of Mathematics and Systems Science, Chinese Academy of Sciences, Beijing, China (\email{keyk@amss.ac.cn})}}

\usepackage{amsopn}

\DeclareMathOperator*{\argmin}{argmin}

\newcommand{\GL}{\mathrm{GL}}

\let\E\undefined
\DeclareMathOperator{\E}{E}
\let\T\undefined
\DeclareMathOperator{\T}{T}
\let\N\undefined
\DeclareMathOperator{\N}{N}
\let\J\undefined
\DeclareMathOperator{\J}{J}

\newcommand{\SN}{\operatorname{SN}}
\newcommand{\SL}{\operatorname{SL}}
\newcommand{\Sp}{\operatorname{Sp}}
\newcommand{\Gr}{\operatorname{Gr}}

\ifpdf
\hypersetup{
  pdftitle={Semi-Riemannian Manifold Optimization},
  pdfauthor={Tingran Gao, Lek-Heng Lim, and Ke Ye}
}
\fi

\begin{document}

\maketitle

% REQUIRED
\begin{abstract}
  We introduce in this paper a manifold optimization framework that utilizes \emph{semi-Riemannian} structures on the underlying smooth manifolds. Unlike in Riemannian geometry, where each tangent space is equipped with a positive definite inner product, a semi-Riemannian manifold allows the metric tensor to be \emph{indefinite} on each tangent space, i.e., possessing both positive and negative definite subspaces; differential geometric objects such as geodesics and parallel-transport can be defined on non-degenerate semi-Riemannian manifolds as well, and can be carefully leveraged to adapt Riemannian optimization algorithms to the semi-Riemannian setting. In particular, we discuss the metric independence of manifold optimization algorithms, and illustrate that the weaker but more general semi-Riemannian geometry often suffices for the purpose of optimizing smooth functions on smooth manifolds in practice.
\end{abstract}

% REQUIRED

% REQUIRED

\section{Introduction}
\label{sec:intro}

Manifold optimization \cite{EAS1998,AMS2009} is a class of techniques for solving optimization problems of the form
\begin{equation}
\label{eq:man-opt-prob}
  \min_{x\in \mathcal{M}}f \left( x \right)
\end{equation}
where $\mathcal{M}$ is a (typically nonlinear and nonconvex) manifold and $f:\mathcal{M}\rightarrow\mathbb{R}$ is a smooth function over $\mathcal{M}$. These techniques generally begin with endowing the manifold $\mathcal{M}$ with a \emph{Riemannian structures}, which amounts to specifying a smooth family of inner products on the tangent spaces of $\mathcal{M}$, with which analogies of differential quantities such as gradient and Hessian can be defined on $\mathcal{M}$ in parallel with their well-known counterparts on Euclidean spaces. This geometric perspective enables us to tackle a constrained optimization problem \cref{eq:man-opt-prob} using methodologies of unconstrained optimization, which becomes particularly beneficial when the constraints (expressed in $\mathcal{M}$) appear highly nonlinear and nonconvex.

The optimization problem \cref{eq:man-opt-prob} is certainly independent of the choice of Riemannian structures on $\mathcal{M}$; in fact, all critical points of $f$ on $\mathcal{M}$ are metric independent. From a differential geometric perspective, equipping the manifold with a Riemannian structure and studying the critical points of a generic smooth function is highly reminiscent of the classical Morse theory \cite{Milnor1963,Nicolaescu2011}, for which the main interest is to understand the topology of the underlying manifold; the topological information needs to be extracted using tools from differential geometry, but is certainly independent of the choice of Riemannian structures. It is thus natural to inquire the influence of different choices of Riemannian metrics on manifold optimization algorithms, which to our knowledge has never been explored in existing literature. This paper stems from our attempts at understanding the dependence of manifold optimization on Riemannian structure. It turns out that most technical tools for optimization on Riemannian manifolds can be extended to a larger class of metric structures on manifolds, namely, \emph{semi-Riemannian structures}. Just as a Riemannian metric is a smooth assignment of inner products to tangent spaces, a semi-Riemannian metric smoothly assigns to each tangent space a \emph{scalar product}, which is a symmetric bilinear form but without the constraint of positive definiteness; our major technical contribution in this paper is an optimization framework built upon the rich differential geometry in such weaker but more general metric structures, of which standard unconstrained optimization on Euclidean spaces and Riemannian manifold optimization are special cases. Though semi-Riemannian geometry has attracted generations of mathematical physicists for its effectiveness in providing space-time model in general relativity \cite{ONeill1983,Carroll2004}, to the best of our knowledge, the link with manifold optimization has never been explored.

A different yet strong motivation for investigating optimization problems on semi-Riemannian manifolds arises from the Riemannian geometric interpretation of interior point methods \cite{NN1994,Renegar2001}. For a twice differentiable and strongly convex function $f$ defined over an open convex domain $Q$ in an Euclidean space, denote by $\nabla f$ and $\nabla^2f$ for the gradient and Hessian of $f$, respectively. The strong convexity of $f$ ensures $\nabla^2 f \left( x \right)\succeq 0$ which defines a \emph{local inner product} $g_x \left( \cdot,\cdot \right):T_xQ\times T_xQ\rightarrow\mathbb{R}$ by
\begin{equation*}
  g_x \left(  v,w \right):=v^{\top}\left[\nabla^2 f \left( x \right)\right]w,\quad\forall v,w\in T_xQ.
\end{equation*}
With respect to this class of new local inner products, which can be interpreted as turning $Q$ into a Riemannian manifold $\left( Q,g \right)$, the gradient of $f$ takes the form
\begin{equation*}
  \tilde{\nabla}f \left( x \right)=\left[ \nabla f \left( x \right) \right]^{-1}\nabla f \left( x \right).
\end{equation*}
The negative manifold gradient $-\tilde{\nabla}f \left( x \right)=-\left[ \nabla f \left( x \right) \right]^{-1}\nabla f \left( x \right)$ coincides with the descent direction $\eta_x$ satisfying the Newton's equation
\begin{equation}
\label{eq:newton-equation}
  \left[\nabla^2 f \left( x \right)\right]\eta_x=-\nabla f \left( x \right)
\end{equation}
at $x\in M$. In other words, the Newton method, which is second order, can be interpreted as a first order method in the Riemannian setting. Such equivalence between first and second order methods under coordinate transformation is also known in other contexts such as \emph{natural gradient descent} in information geometry; see \cite{RM2015} and the references therein. Extending this geometric picture beyond the relatively well-understood case of strongly convex functions requires understanding optimization on semi-Riemannian manifolds as a first step; we expect the theoretical foundation laid out in this paper will shed light upon gaining deeper geometric insights on the convergence of non-convex optimization algorithms.

The rest of this paper is organized as follows. In \Cref{sec:preliminaries} we provide a brief but self-contained introduction to Riemannian optimization and semi-Riemannian geometry. \Cref{sec:semi-riem-optim} details the algorithmic framework of semi-Riemannian optimization, and proposes semi-Riemannian analogies of the Riemannian steepest descent and conjugate gradient algorithms; the metric independence of some second-order algorithms are also investigated. We specialize the general geometric framework to submanifolds in \Cref{sec:semi-riem-geom-sub}, in which we characterize the phenomenon (which does not exist in Riemannian geometry) of \emph{degeneracy} for induced semi-Riemannian structures, and identify several (nearly) non-degenerate examples to which our general algorithmic framework applies. We illustrate the utility of the proposed framework with several examples in \Cref{sec:numer-exper} and conclude with \Cref{sec:discussion}. More examples and some omitted proofs are deferred to the Supplementary Materials.

\section{Preliminaries}
\label{sec:preliminaries}

\subsection{Notations}
\label{sec:notations}

We denote a smooth manifold using $M$ or $\mathcal{M}$. Lower case letters such as $a,b,c$ or $x,y,z$ will be used to denote vectors or points on a manifold, depending on the context. We write $T\!M$ and $T^{*}M$ for the tangent and cotangent bundles of $M$, respectively. For a fibre bundle $E$, $\Gamma \left( E \right)$ will be used to denote smooth sections of this bundle. Unless otherwise specified, we use $\left\langle \cdot,\cdot \right\rangle$ or $g\in \Gamma \left( T^{*}\!M\otimes T^{*}\!M \right)$ to denote a semi-Riemannian metric. For a smooth function $f$, notations $Df$ and $D^2f$ stand for semi-Riemannian gradients and Hessians, respectively, when they exist; $\nabla f$ and $\nabla^2f$ will be reserved for Riemannian gradients and Hessians, respectively. More generally, $D$ will be used to denote the Levi-Civita connection on the semi-Riemannian manifold, while $\nabla$ denotes for the Levi-Civita connection on a Riemannian manifold. We denote anti-symmetric (i.e. skew-symmetric) matrices and symmetric matrices of size $n$-by-$n$ with $\mathrm{Skew}\left( \mathbb{R}^{n\times n} \right)$ and $\mathrm{Sym}\left( \mathbb{R}^{n\times n} \right)$, respectively. For a vector space $V$, $\bigwedge^k V$ and $S^kV$ stands for alternated or symmetrized $k$ copies of $V$, respectively.

\subsection{Riemannian Manifold Optimization}
\label{sec:riem-manif-optim}

As stated at the beginning of this paper, manifold optimization is a type of nonlinear optimization problems taking the form of \cref{eq:man-opt-prob}. The methodology of Riemannian optimization is to equip the smooth manifold $M$ with a Riemannian metric structure, i.e. positive definite bilinear forms $\left\langle \cdot,\cdot \right\rangle$ on the tangent spaces of $\mathcal{M}$ that varies smoothly on the manifold \cite{Milnor1973,doCarmo1992RG,Petersen2006}. The differentiable structure on $\mathcal{M}$ facilitates generalizing the concept of differentiable functions from Euclidean spaces to these nonlinear objects; in particular, notions such as gradient and Hessian are available on Riemannian manifolds and play the same role as their Euclidean space counterparts.

The algorithmic framework of Riemannian manifold optimization has been established and investigated in a sequence of works \cite{Gabay1982,Smith1994,EAS1998,AMS2009}. These algorithms typically builds upon the concepts of \emph{gradient}, the first-order differential operator $\nabla:C^1 \left( \mathcal{M} \right)\rightarrow \Gamma \left( TM \right)$ defined by
\begin{equation*}
  \left\langle \nabla f \left( x \right), X \right\rangle = Xf \left( x \right)\quad \forall X\in T_xM,
\end{equation*}
and \emph{Hessian}, the covariant derivative of the gradient operator defined by
\begin{equation*}
  \nabla^2f \left( X,Y \right)=XYf - \left(\nabla_XY\right)f\quad\forall X,Y\in \Gamma \left( TM \right)
\end{equation*}
as well as a \emph{retraction} $\mathrm{Retr}_x:T_x\mathcal{M}\rightarrow \mathcal{M}$ from each tangent plane $T_x\mathcal{M}$ to the manifold $\mathcal{M}$ such that (1) $\mathrm{Retr}_x \left( 0 \right)=x$ for all $x\in\mathcal{M}$, and (2) the differential map of $\mathrm{Retr}_x$ is identify at $0\in T_x\mathcal{M}$. On Riemannian manifolds it is natural to use the exponential mapping as the retraction, but any general map from tangent spaces to the Riemannian manifold suffices; in fact, the only requirement implied by conditions (1) and (2) is that the retraction map coincides with the exponential map up to the first order.

The optimality conditions for unconstrained optimization on Euclidean spaces in terms of gradients and Hessians can be naturally translated into the Riemannian manifold setting:%gradient and Hessian in characterizing the optimality conditions for problems of type \cref{eq:man-opt-prob} is essentially a literal translation of the Euclidean setup:
\begin{proposition}[\cite{BAC2016}, Proposition 1.1]
  \label{prop:riem-man-opt-optimality}
  A local optimum $x\in\mathcal{M}$ of Problem \cref{eq:man-opt-prob} satisfies the following necessary conditions:
\begin{enumerate}[(i)]
\item\label{item:1} $\nabla f \left( x \right)=0$ if $f:\mathcal{M}\rightarrow \mathbb{R}$ is first-order differentiable;
\item\label{item:2} $\nabla f \left( x \right)=0$ and $\nabla^2f \left( x \right)\succeq 0$ if $f:\mathcal{M}\rightarrow \mathbb{R}$ is second-order differentiable.
\end{enumerate}
\end{proposition}
Following \cite{BAC2016}, we call $x\in\mathcal{M}$ satisfying condition (i) in \cref{prop:riem-man-opt-optimality} a \emph{(first-order) critical point} or \emph{stationary point}, and a point satisfying condition (i) in \cref{prop:riem-man-opt-optimality} a \emph{second-order critical point}.

The heart of Riemannian manifold optimization is to transform the nonlinear constrained optimization problem \cref{eq:man-opt-prob} into an unconstrained problem on the manifold $\mathcal{M}$. Following this methodology, classical unconstrained optimization algorithms such as gradient descent, conjugate gradients, Newton's method, and trust region methods have been generalized to Riemannian manifolds; see \cite[Chapter 8]{AMS2009}. For instance, the dynamics of the iterates $x_0,x_1,\cdots,x_k,\cdots$ generated by gradient descent algorithm on Riemannian manifolds essentially replaces the descent step $x_{k+1}=x_k-\nabla f \left( x_k \right)$ with its Riemannian counterpart $x_{k+1}=\mathrm{Retr}_{x_k}\left( -\nabla f \left( x_k \right) \right)$. Other differential geometric objects such as parallel-transport, Hessian, and curvature render themselves naturally en route to adapting other unconstrained optimization algorithms to the manifold setting. We refer interested readers to \cite{AMS2009} for more details.

\subsection{Semi-Riemannian Geometry}
\label{sec:semi-riem-geom}

Semi-Riemannian geometry differs from Riemannian geometry in that the bilinear form equipped on each tangent space can be indefinite. Classical examples include Lorentzian spaces and De Sitter spaces in general relativity; see e.g. \cite{ONeill1983,Carroll2004}. Although one may think of Riemannian geometry as a special case of semi-Riemannian geometry as all Riemannian metric tensors are automatically semi-Riemannian, the existence of a semi-Riemannian metric with nontrivial \emph{index} (see definition below) actually imposes additional constraints on the tangent bundle of the manifold and is thus often more restrictive---the tangent bundle should admit a non-trivial splitting into the direct sum of ``positive definite'' and ``negative definite'' sub-bundles. Nevertheless, such metric structures have found vast applications in and beyond understanding the geometry of spacetime, for instance, in the study of the regularity of optimal transport maps \cite{KMW2010,KM2010,AKM2011}.

\begin{definition}
A symmetric bilinear form $\left\langle\cdot,\cdot \right\rangle:V\times V\rightarrow \mathbb{R}$ on a vector space $V$ is non-degenerate if
\begin{equation*}
    \left\langle v,w\right\rangle =0\,\,\textrm{for all}\,\,w\in V\quad\Leftrightarrow\quad v=0.
  \end{equation*}
The \emph{index} $\nu\in\mathbb{Z}_{\geq 0}$ of a symmetric bilinear form on $V$ is the dimension of the maximum negative definite subspace of $V$; similarly, we denote $\pi\in\mathbb{Z}_{\geq 0}$ for the dimension of the maximum positive definite subspace of $V$. A \emph{scalar product} on a vector space $V$ is a non-degenerate symmetric bilinear form on $V$. The \emph{signature} of a scalar product on $V$ with index $\nu$ is a vector of length $\mathrm{dim} \left( V \right)$ with the first $\nu$ entries equaling $-1$ and the rest of entries equaling $1$. A subspace $W\subset V$ is said to be \emph{non-degenerate} if the restriction of the scalar product to $W$ is non-degenerate.
\end{definition}
The main difference between a scalar product and an inner product is that the former needs not possess positive definiteness. The main issue with this lack of positivity is the consequent lack of a meaningful definition for ``orthogonality'' --- a vector subspace may well be the orthogonal complement of itself: consider for example the subspace spanned by $\left( 1,1 \right)$ in $\mathbb{R}^2$ equipped with a scalar product with signature $\left( -, + \right)$. The same example illustrates that the property of non-degeneracy is not always inheritable by subspaces. Nonetheless, the following is true:
\begin{lemma}[Chapter 2, Lemma 23, \cite{ONeill1983}]
\label{lem:nondegeneracy}
  A subspace $W$ of a vector space $V$ is non-degenerate if and only if $V=W\oplus W^{\perp}$.
\end{lemma}
\begin{definition}[Semi-Riemannian Manifolds]
  A \emph{metric tensor} $g$ on a smooth manifold $M$ is a symmetric non-degenerate $\left( 0,2 \right)$ tensor field on $M$ of constant index. A \emph{semi-Riemannian manifold} is a smooth manifold $M$ equipped with a metric tensor.
\end{definition}

\begin{example}[Minkowski Spaces $\mathbb{R}^{p,q}$]
  \label{exm:euc-pq}
  Consider the Euclidean space $\mathbb{R}^n$ and denote $\I_{p,q}$ for the $n$-by-$n$ diagonal matrix with the first $p$ diagonal entries equaling $-1$ and the rest $q=n-p$ entries equaling $1$, where $0\leq p\leq n$ and $n\geq 1$. For arbitrary $u,w\in\mathbb{R}^n$, define the bilinear form
  \begin{equation*}
    \left\langle u,v \right\rangle:=u^{\top}\I_{p,q} w.
  \end{equation*}
It is straightforward to verify that this bilinear form is nondegenerate on $\mathbb{R}^n$, and that such defined $\left( \mathbb{R}, \left\langle \cdot,\cdot \right\rangle \right)$ is a semi-Riemannian manifold. This space is known as the \emph{Minkowski space of signature $\left( p,q \right)$}.
\end{example}

\begin{example}
  \label{exm:matrix-semi-riem}
  Consider the vector space of matrices $\mathbb{R}^{n\times n}$, where $n\in\mathbb{N}$ and $n=p+q$, $p,q\in\mathbb{N}$. Define a bilinear form on $\mathbb{R}^{n\times n}$ by
  \begin{equation*}
    \left\langle A,B \right\rangle := \mathrm{Tr}\left( A^{\top}\I_{p,q}B \right),\quad\forall A,B\in \mathbb{R}^{n\times n}.
  \end{equation*}
  This bilinear form is non-degenerate on $\mathbb{R}^{n\times n}$, because for any $A,B\in\mathbb{R}^{n\times n}$ we have
  \begin{equation*}
    \mathrm{Tr}\left( A^{\top}\I_{p,q}B \right)=\mathrm{vec}\left( A \right)^{\top} \left( I_n\otimes I_{p,q} \right)\mathrm{vec}\left( B \right)
  \end{equation*}
  where $I_n$ is the identity matrix of size $n$-by-$n$, $\otimes$ denotes for the Kronecker product, and $\mathrm{vec}:\mathbb{R}^{n\times n}\rightarrow \mathbb{R}^{n^2}$ is the \emph{vectorization} operator that vertically stacks the columns of a matrix in $\mathbb{R}^{n\times n}$. The non-degeneracy then follows from \cref{exm:euc-pq}. This example gives rise to a semi-Riemannian structure for matrices in $\mathbb{R}^{n\times n}$.
\end{example}

The non-degeneracy of the semi-Riemannian metric tensor ensures that most classical constructions on Riemannian manifolds have their analogies on a semi-Riemannian manifold. Most fundamentally, the ``miracle of Riemannian geometry'' --- the existence and uniqueness of a canonical connection --- is beheld on semi-Riemannian manifolds as well. Quoting \cite[Theorem 11]{ONeill1983}, on a semi-Riemannian manifold $M$ there is a unique connection $D:\Gamma \left( M,TM \right)\rightarrow\Gamma \left( M,T^{\otimes 2}M \right)$ such that
\begin{equation}\label{eqn:torsion free}
  \left[ V,W \right]=D_VW-D_WV
\end{equation}
and
\begin{equation}\label{eqn:metric compactible}
  X \left\langle V,W \right\rangle = \left\langle D_XV, W \right\rangle+\left\langle V,D_XW \right\rangle
\end{equation}
for all $X,V,W\in \Gamma \left( M,TM \right)$. This connection is called the \emph{Levi-Civita connection} of $M$ and is characterized by the \emph{Koszul formula}
\begin{equation}\label{eqn:Koszul formula}
  \begin{aligned}
    2 \left\langle D_VW, X \right\rangle = &V \left\langle W,X \right\rangle+ W \left\langle X,V \right\rangle-X \left\langle V,W \right\rangle\\
    &-\left\langle V,\left[ W,X \right] \right\rangle+\left\langle W,\left[ X,V \right] \right\rangle+\left\langle X,\left[ V,W \right] \right\rangle\quad \forall X,V,W\in \Gamma \left( M,TM \right).
  \end{aligned}
\end{equation}
Geodesics, parallel-transport, and curvature of $M$ can be defined via the Levi-Civita connection on $M$ in an entirely analogous manner as on Riemannian manifolds.

Differential operators can be defined on semi-Riemannian manifolds much the same way as on Riemannian manifolds. For any $f\in C^1 \left( M \right)$, where $M$ is a semi-Riemannian manifold, the \emph{gradient} of $f$, denoted as $D f\in \Gamma \left( M,TM \right)$, is defined by the equality (c.f. \cite[Definition 47]{ONeill1983})
\begin{equation}
  \label{eq:semi-riem-grad}
  \left\langle D f,X \right\rangle=Xf,\quad\forall X\in\Gamma \left( M,TM \right).
\end{equation}
The \emph{Hessian} of $f\in C^2 \left( M \right)$ can be similarly defined, also similar to the Riemannian case (\cite[Definition 48, Lemma 49]{ONeill1983}), by $D^2f=D \left(D f\right)\in\Gamma \left( M,T^{*}M\otimes T^{*}M \right)$, or equivalently
\begin{equation}
  \label{eq:semi-riem-hess}
  D^2f \left( X,Y \right)=XYf-\left( D_XY \right)f,\quad\forall X,Y\in \Gamma \left( M,TM \right).
\end{equation}
Since the Levi-Civita connection on $M$ is torsion-free, $\nabla^2f$ is a symmetric $\left( 0,2 \right)$ tensor field on $M$, i.e.,
\begin{equation*}
  D^2f \left( X,Y \right)=D^2 f \left( Y,X \right),\quad\forall X,Y\in \Gamma \left( M,TM \right).
\end{equation*}

One way to compare the semi-Riemannian and Riemannian gradients and Hessians, when both metric structures exist on the same smooth manifold, is through their local coordinate expressions. In fact, the local coordinate expressions for the two types (Riemannian/semi-Riemannian) of differential operators can be unified as follows. Let $\left\{ x^1,\cdots,x^n \right\}$ be a local coordinate system around an arbitrary point $x\in\mathcal{M}$, and denote $g_{ij}$ and $h_{ij}$ for the components of the Riemannian and semi-Riemannian metric tensors, respectively; the Christoffel symbols will be denoted as $\phantom{}^g\Gamma_{ij}^{k}$ and $\phantom{}^h\Gamma_{ij}^{k}$, respectively. Direct computation reveals
\begin{equation}
  \label{eq:riem-semiriem-grad-hess}
  \begin{aligned}
    &\nabla f=g^{ij}\partial_jf\partial_i, \qquad \nabla^2f=\left(\partial_{ij}^2f-\phantom{}^g\Gamma_{ij}^k \partial_kf\right)\mathrm{d}x^i\otimes \mathrm{d}x^j,\\
    &D f=h^{ij}\partial_jf\partial_i, \qquad D^2f=\left(\partial_{ij}^2f-\phantom{}^h\Gamma_{ij}^k \partial_kf\right)\mathrm{d}x^i\otimes\mathrm{d}x^j.
  \end{aligned}
\end{equation}
Using the music isomorphism induced from the (Riemannian or semi-Riemannian) metric, the Hessians can be cast in the form of $\left( 2,0 \right)$-tensors on $\Gamma \left( T\!M\otimes T\!M \right)$ as
\begin{equation*}
  \begin{aligned}
    \left( \nabla^2f \right)^{\sharp}&=g^{i\ell}g^{jm}\left(\partial_{ij}^2f-\phantom{}^g\Gamma_{ij}^k \partial_kf\right)\partial_i\otimes \partial_m,\\
    \left( D^2f \right)^{\sharp}&=h^{i\ell}h^{jm}\left(\partial_{ij}^2f-\phantom{}^h\Gamma_{ij}^k \partial_kf\right)\partial_i\otimes \partial_m.
  \end{aligned}
\end{equation*}

\begin{remark}
  \label{rem:hessian-equiv}
  Notably, for any $x\in\mathcal{M}$, if we compute the Hessians $D^2f \left( x \right)$ and $\nabla^2f \left( x \right)$ in the corresponding geodesic normal coordinates centered at $x$, \cref{eq:riem-semiriem-grad-hess} implies that the two Hessians take the same coordinate form $\left(\partial_{ij}^2f\right)_{1\leq i,j\leq n}$ since both $\phantom{}^g\Gamma_{ij}^k$ and $\phantom{}^h\Gamma_{ij}^k$ vanish at $x$. For instance, $\mathbb{R}^n$ has the same geodesics under the Euclidean or Lorentzian metric (straight lines), and the standard coordinate system serves as geodesic normal coordinate system for both metrics; see \cref{exm:semi-riem-euc}. In particular, the notion of geodesic convexity \cite{Rapcsak1991,Udriste1994} is equivalent for the two different of metrics; this equivalence is not completely trivial by the well-known first and second order characterization (see e.g. \cite[Theorem 5.1]{Udriste1994} and \cite[Theorem 6.1]{Udriste1994}) since geodesics need not be the same under different metrics.
\end{remark}

\begin{proposition}
\label{prop:equivalence-geo-convexity}
  On a smooth manifold $\mathcal{M}$ admitting two different Riemannian or semi-Riemannian structures, an optimization problem is geodesic convex with respect to one metric if and only if it is also geodesic convex with respect to another.
\end{proposition}
\begin{proof}
  Denote the two metric tensors on $\mathcal{M}$ as $g$ and $h$, respectively. Both $g$ and $h$ can be Riemannian or semi-Riemannian, respectively or simultaneously. For any $x\in \mathcal{M}$, let $x^1,\cdots,x^n$ and $y^1,\cdots,y^n$ be the geodesic coordinates around $x$ with respect to $g$ and $h$, respectively. Denote $J=\left( \partial y_j/\partial x_i \right)_{1\leq i,j\leq n}$ for the Jacobian of the coordinate transformation between the two normal coordinate systems. The coordinate expressions of a tangent vector $v\in T_x\mathcal{M}$ in the two normal coordinate systems are linked by (Einstein summation convention adopted)
  \begin{equation*}
    v=v^i \partial/\partial x_i = \tilde{v}^j \partial/ \partial y_j\quad\Leftrightarrow\quad v^i=\tilde{v}^j\partial x_i/\partial y_j.
  \end{equation*}
  Therefore
  \begin{equation*}
    \begin{aligned}
      &\left[\nabla^2f \left( x \right)\right] \left( v,v \right)\geq 0\quad\forall v\in T_x\mathcal{M}\\
      \Leftrightarrow \quad & v^iv^j \frac{\partial^2f}{\partial x_i\partial x_j} \left( x \right)\geq 0\quad \forall v^1,\cdots,v^n\in\mathbb{R}\\
      \Leftrightarrow \quad & \tilde{v}^{\ell} \frac{\partial x_i}{\partial y_{\ell}}\tilde{v}^m \frac{\partial x_j}{\partial y_m} \frac{\partial^2f}{\partial y_i\partial y_j}\geq 0\quad \forall \tilde{v}^1,\cdots,\tilde{v}^n\in\mathbb{R}\\
      \Leftrightarrow \quad &\left[D^2f \left( x \right)\right] \left( v,v \right)\geq 0\quad\forall v\in T_x\mathcal{M}.
    \end{aligned}
  \end{equation*}
  which establishes the desired equivalence.
\end{proof}

\begin{example}[Gradient and Hessian in Minkowski Spaces]
  \label{exm:semi-riem-euc}
  Consider the Euclidean space $\mathbb{R}^n$. Denote $I_{p,q}\in\mathbb{R}^{n\times n}$ for the $n$-by-$n$ diagonal matrix with the first $p$ diagonal entries equaling $-1$ and the rest $q=n-p$ diagonal entries equaling $1$. We compute and compare in this example the gradients and Hessians of differentiable functions on $\mathbb{R}^n$. We take the Riemannian metric as the standard Euclidean metric, and the semi-Riemannian metric given by $\I_{p,q}$. For any $f\in C^2 \left( M \right)$, the gradient of $f$ is determined by
  \begin{equation*}
    \begin{aligned}
      \left( Df \right)^{\top}\I_{p,q}X&=Xf=\left( \nabla f \right)^{\top}X,\quad\forall X\in \Gamma \left( \mathbb{R}^n, \mathbb{R}^n \right)\\
      &\Leftrightarrow Df=\I_{p,q}\nabla f\qquad\textrm{where $\nabla f =\left( \partial_1f,\cdots,\partial_nf \right)\in \mathbb{R}^n$.}
    \end{aligned}
  \end{equation*}
  Furthermore, since in this case the semi-Riemannian metric tensor is constant on $\mathbb{R}^n$, the Christoffel symbol vanishes (c.f. \cite[Chap 3. Proposition 13 and Lemma 14]{ONeill1983}), and thus $D_XDf=\I_{p,q}\nabla_X\nabla f=\I_{p,q}\left(\nabla^2f\right)X$ for all $X\in \Gamma \left( \mathbb{R}^n, \mathbb{R}^n \right)$, where $$\nabla^2f=\left( \partial_i\partial_jf \right)_{1\leq i,j\leq n}\in\mathbb{R}^{n\times n}.$$
By the definition of Hessian, for all $X,Y\in\Gamma \left( \mathbb{R}^n,\mathbb{R}^n \right)$ we have
  \begin{equation*}
    D^2f \left( X,Y \right)=\left\langle D_XDf,Y \right\rangle=Y^{\top}\I_{p,q}\cdot\I_{p,q}\left(\nabla^2f\right)X=Y^{\top}\left(\nabla^2f\right)X
  \end{equation*}
  from which we deduce the equality $D^2f=\nabla^2f$. In fact, the equivalence of the two Hessians also follows directly from \cref{rem:hessian-equiv}, since the geodesics under the Riemannian and semi-Riemannian metrics coincide in this example (see e.g. \cite[Chapter 3 Example 25]{ONeill1983}). In particular, the equivalence between the two types of geodesics and Hessians imply the equivalence of geodesic convexity for the two metrics.
\end{example}

\section{Semi-Riemannian Optimization Framework}
\label{sec:semi-riem-optim}

This section introduces the algorithmic framework of semi-Riemannian optimization. To begin with, we point out that the first- and second-order necessary conditions for optimality in unconstrained optimization and Riemannian optimization can be directly generalized to semi-Riemannian manifolds. We then generalize several Riemannian manifold optimization algorithms to their semi-Riemannian counterparts, and illustrate the difference with a few numerical examples. We end this section by showing global and local convergence results for semi-Riemannian optimization.

\subsection{Optimality Conditions}
\label{sec:optim-cond}

The following \cref{prop:semi-riem-man-opt-optimality} should be considered as the semi-Riemannian analogy of the optimality conditions \cref{prop:riem-man-opt-optimality}
.
\begin{proposition}[Semi-Riemannian First-  and Second-Order Necessary Conditions for Optimality]
  \label{prop:semi-riem-man-opt-optimality}
  Let $\mathcal{M}$ be a semi-Riemannian manifold. A local optimum $x\in\mathcal{M}$ of Problem \cref{eq:man-opt-prob} satisfies the following necessary conditions:
\begin{enumerate}[(i)]
\item\label{item:3} $Df \left( x \right)=0$ if $f:\mathcal{M}\rightarrow \mathbb{R}$ is first-order differentiable;
\item\label{item:4} $Df \left( x \right)=0$ and $D^2f \left( x \right)\succeq 0$ if $f:\mathcal{M}\rightarrow \mathbb{R}$ is second-order differentiable.
\end{enumerate}
\end{proposition}
\begin{proof}
\begin{enumerate}[(i)]
\item\label{item:9} If $x\in\mathcal{M}$ is a local optimum of \cref{eq:man-opt-prob}, then for any $X\in\Gamma \left( TM \right)$ we have $Xf \left( x \right)=0$, which, by definition \cref{eq:semi-riem-grad} and the non-degeneracy of the semi-Riemannian metric, implies that $Df \left( x \right)=0$.
\item\label{item:11} If $x\in\mathcal{M}$ is a local optimum of \cref{eq:man-opt-prob}, then there exists a local neighborhood $U\subset\mathcal{M}$ of $x$ such that $f \left( y \right)\geq f \left( x \right)$ for all $y\in U$. Without loss of generality we can assume that $U$ is sufficiently small so as to be geodesically convex (see e.g. \cite[\S3.4]{doCarmo1992RG}). Denote $\gamma:\left[ -1,1 \right]\rightarrow U$ for a constant-speed geodesic segment connecting $\gamma \left( 0 \right)=x$ to $\gamma \left( 1 \right)=y$ that lies entirely in $U$. The one-variable function $t\mapsto f\circ\gamma \left( t \right)$ admits Taylor expansion
  \begin{equation*}
    \begin{aligned}
      f \left( y \right) &= f \circ \gamma \left( 1 \right)=f\circ\gamma \left( 0 \right)+\left( f\circ\gamma \right)' \left( 0 \right)+\frac{1}{2}\left( f\circ\gamma \right)'' \left( \xi \right)\\
      &=f \left( x \right)+\left\langle Df \left( x \right), \gamma' \left( 0 \right) \right\rangle+\frac{1}{2}D_{\gamma' \left( \xi \right)}\left\langle Df \left( \gamma \left( \xi \right) \right), \gamma' \left( \xi \right) \right\rangle\\
      &=f \left( x \right)+\frac{1}{2}\left[D^2f \left( \gamma \left( \xi \right) \right)\right] \left( \gamma' \left( \xi \right), \gamma' \left( \xi \right) \right)
    \end{aligned}
  \end{equation*}
  where the last equality used $Df \left( x \right)=0$. Letting $y\rightarrow x$ on $\mathcal{M}$, the smoothness of $D^2f$ ensures that
  \begin{equation*}
    D^2f \left( x \right)\left[ V,V \right]\geq 0\qquad\forall V\in T_x\mathcal{M}
  \end{equation*}
  which establishes $D^2f \left( x \right)\succeq 0$.
\end{enumerate}
\end{proof}

The formal similarity between \cref{prop:semi-riem-man-opt-optimality} and \cref{prop:riem-man-opt-optimality} is not entirely surprising. As can be seen from the proofs, both optimality conditions are based on geometric interpretations of the same Taylor expansion; the metrics affect the specific forms of the gradient and Hessian, but the optimality conditions are essentially derived from the Taylor expansions only. Completely parallel to the Riemannian setting, we can also translate the second-order sufficient conditions \cite[\S7.3]{LuenbergerYe1984} into the semi-Riemannian setting without much difficulty. The proof essentially follows \cite[\S7.3 Proposition 3]{LuenbergerYe1984}, with the Taylor expansion replaced with the expansion along geodesics in \cref{prop:semi-riem-man-opt-optimality} (ii); we omit the proof since it is straightforward, but document the result in \cref{prop:semi-riem-man-opt-second-optimality} below for future reference. Recall from \cite[\S7.1]{LuenbergerYe1984} that $x\in\mathcal{M}$ is a strict relative minimum point of $f$ on $\mathcal{M}$ if there is a local neighborhood of $x$ on $\mathcal{M}$ such that $f \left( y \right)>f \left( x \right)$ for all $y\in U\backslash \left\{ x \right\}$.

\begin{proposition}[Semi-Riemannian Second-Order Sufficient Conditions]
  \label{prop:semi-riem-man-opt-second-optimality}
  Let $f$ be a second differentiable function on a semi-Riemannian manifold $\mathcal{M}$, and $x\in \mathcal{M}$ is a an interior point. If $Df \left( x \right)=0$ and $D^2f \left( x \right)\succ 0$, then $x$ is a strict relative minimum point of $f$.
\end{proposition}

The formal similarity between the Riemannian and semi-Riemannian optimality conditions indicates that it might be possible to transfer many technologies in manifold optimization from the Riemannian to the semi-Riemannian setting. For instance, the equivalence of the first-order necessary condition implies that, in order to search for a first-order stationary point, on a semi-Riemannian manifold we should look for points at which the semi-Riemannian gradient $Df$ vanishes, just like in the Riemannian realm we look for points at which the Riemannian gradient $\nabla f$ vanishes. However, extra care has to be taken regarding the influence different metric structures have on the induced topology of the underlying manifold. For Riemannian manifolds, it is straightforward to check that the induced topology coincides with the original topology of the underlying manifold (see e.g. \cite[Chap 7 Proposition 2.6]{doCarmo1992RG}), whereas the ``topology'' induced by a semi-Riemannian structure is generally quite pathological --- for instance, two distinct points connected by a light-like geodesic (a geodesic along which all tangent vectors are \emph{null vectors} (c.f. \cref{defn:vector-types})) has zero distance. An exemplary consequence is that, in search of a first-order stationary point, we shouldn't be looking for points at which $\left\| Df \right\|^2$ vanishes since this does not imply $Df=0$.

\subsection{Determining the ``Steepest Descent Direction''}
\label{sec:algorithms}

As long as gradients, Hessians, retractions, and parallel-transports can be properly defined, one might think there exists no essential difficulty in generalizing any Riemannian optimization algorithms to the semi-Riemannian setup, with the Riemannian geometric quantities replaced with their semi-Riemannian counterparts, \emph{mutatis mutandis}. It is tempting to apply this methodology to all standard manifold optimization algorithms, including but not limited to first-order methods such as steepest descent, conjugate gradient descent, and quasi-Newton methods, or second-order methods such as Newton's method and trust region methods. We discuss in this subsection how to determine a proper descent direction for steepest-descent-type algorithms on a semi-Riemannian manifold. Some exemplary first- and second-order methods will be discussed in the next subsection.

%; the first-order methods, for which the key ingredient %See \cref{alg:gd-semi-mfld} for an example that adapts the Riemannian steepest descent algorithm to the semi-Riemannian setting.

%\emph{Steepest Descent.}

As one of the prototypical first-order optimization algorithms, gradient descent is known for its simplicity yet surprisingly powerful theoretical guarantees under mild technical assumptions. A plausible ``Semi-Riemannian Gradient Descent'' algorithm that na{\"i}vely follows the paradigm of Riemannian gradient descent could be designed as simply replacing the Riemannian gradient $\nabla f$ with the semi-Riemannian gradient $Df$ defined in \cref{eq:semi-riem-grad}, as listed in \cref{alg:gd-semi-mfld}. Of course, a key step in \cref{alg:gd-semi-mfld} is to determine the descent direction $\eta_k$ in each iteration. However, while negative gradient is an obvious choice in Riemannian manifold optimization, the ``steepest descent direction'' is a slightly more subtle notion in semi-Riemannian geometry, as will be demonstrated shortly in this section.

A first difficulty with replacing $-\nabla f \left( x \right)$ by $-Df \left( x \right)$ is that $-Df \left( x \right)$ needs not be a descent direction at all: consider, for instance, an illustrative example of optimization in the Minkowski space (Euclidean space equipped with the standard semi-Riemannian metric): the first order Taylor expansion at $x$ gives for any small $t>0$
\begin{equation}
  \label{eq:taylor-expansion}
  f \left( x-t D f \left( x \right) \right) \approx f \left( x \right)-t \left\langle Df \left( x \right), Df \left( x \right) \right\rangle %= f \left( x \right)-t \left\| Df \left( x \right) \right\|^2
\end{equation}
but in the semi-Riemannian setting the scalar product term $\left\langle Df \left( x \right), Df \left( x \right) \right\rangle$ may well be negative, unlike the Riemannian case. In order for the value of the objective function to decrease (at least in the first order), we have to pick the descent direction to be either $Df \left( x \right)$ or $-Df \left( x \right)$, whichever makes $\left\langle Df \left( x \right), Df \left( x \right) \right\rangle>0$. %With this quick fix, and paired with an efficient line search strategy such as Armijo's rule \cite[\S8.5]{LuenbergerYe1984}, one can actually cast the Riemannian gradient descent algorithm completely into the semi-Riemannian setup, which, when implemented, actually works for quite a wide range of generic problems. 

Though the quick fix by replacing $Df \left( x \right)$ with $\pm Df \left( x \right)$ would work generically in many problems of practical interest, a second, and more serious issue with choosing $\pm Df \left( x \right)$ as the descent direction lies inherently at the indefiniteness of the metric tensor. For standard gradient descent algorithms (e.g. on Euclidean spaces with standard metric, or more generally on Riemannian manifolds), the algorithm terminates after $\left\| \nabla f \right\|$ becomes smaller than a predefined threshold; for norms induced from positive definite metric tensors, $\left\| \nabla f \right\|\approx 0$ is equivalent to characterizing $\nabla f\approx 0$, implying that the sequence $\left\{ x_k\mid k=0,1,\cdots \right\}$ is truly approaching a first order stationary point. This intuition breaks down for indefinite metric tensors as $\left\| D f \right\|\approx 0$ no longer implies the proximity between $Df$ and $0$. Even though one can fix this ill-defined termination condition by introducing an auxiliary Riemannian metric (which always exists on a Riemannian manifold), when $Df$ is a null vector (i.e. $\left\| Df \right\|=0$, see \cref{defn:vector-types}), the gradient algorithm loses the first order decrease in the objection function value (see \cref{eq:taylor-expansion}); the validity of the algorithm then relies upon second-order information, with which we lose the benefits of first-order methods. As a concrete example, consider the unconstrained optimization problem on the Minkowski space $\mathbb{R}^2$ equipped with a metric of signature $\left( -1,1 \right)$:
\begin{equation*}
  \min_{x,y\in\mathbb{R}}f \left( x,y \right)=\frac{1}{2}\left( x-y \right)^2.
\end{equation*}
Recall from \cref{exm:semi-riem-euc} that
\begin{equation*}
  Df \left( x,y \right)=\I_{1,1}\nabla f \left( x,y \right)=-\left( x-y \right) \cdot \left( 1,1 \right)^{\top}
\end{equation*}
which is a direction parallel to the isolines of the objective function $f$. Thus the semi-Riemannian gradient descent will never decrease the objective function value.

\begin{algorithm}
\caption{{\sc Semi-Riemannian Steepest Descent}}
\label{alg:gd-semi-mfld}
\begin{algorithmic}[1]
  \Require{Manifold $M$, semi-Riemannian metric $\left\langle \cdot, \cdot \right\rangle$, objective function $f$, retraction $\mathrm{Retr}:T\!M\rightarrow M$, initial value $x_0\in M$, parameters for {\sc Linesearch}, gradient $Df$}
  \State $x_0\gets$ {\sc Initiate}
  \State $k\gets 0$
  \While{not converge}
  \State $\eta\gets${\sc FindDescentDirection}$\left( x_k,M,Df \left( x_k \right) \right)$ \Comment{c.f. \cref{alg:finding-steepest-descent-dir}}
  \State $0<t_k\gets$ {\sc LineSearch}$\left( f, x_k, \eta_k \right)$ \Comment{$t_k$ is the Armijo step size}
  \State Choose $x_{k+1}$ such that \Comment{$c\in \left( 0,1 \right)$ is a parameter}
  \begin{equation*}
   f \left( x_k \right)-f \left( x_{k+1} \right) > c\left[f \left( x_k \right) - f\left(\mathrm{Retr}_{x_k} \left( t_k\eta_k \right)\right)\right]
  \end{equation*}
  \State $k\gets k+1$
  \EndWhile
  \State \Return Sequence of iterates $\left\{ x_k \right\}$
 \end{algorithmic}
\end{algorithm}

To rectify these issues, it is necessary to revisit the motivating, geometric interpretation of the negative gradient direction as the direction of ``steepest descent,'' i.e. for any Riemannian manifold $\left( M,g \right)$ and function $f$ on $M$ differentiable at $x\in M$, we know from vector arithmetic that
\begin{equation}
  \label{eq:riem-steepest-descent}
  -\frac{\nabla f \left( x \right)}{\sqrt{g \left( \nabla f \left( x \right), \nabla f \left( x \right) \right)}} =\argmin _{V\in T_xM\atop g\left( V,V \right)=1}g \left( V,\nabla f \left( x \right) \right)=\argmin _{V\in T_xM\atop g\left(V, V \right)=1}Vf \left( x \right).
\end{equation}
In the semi-Riemannian setting, assuming $M$ is equipped with a semi-Riemannian metric $h$, we can also set the descent direction leading to the steepest decrease of the objective function value. It is not hard to see that in general
\begin{equation}
  \label{eq:semi-riem-steepest-descent}
  \pm \frac{Df \left( x \right)}{\sqrt{\left|h\left( Df \left( x \right), Df \left( x \right) \right)\right|}} \neq \argmin _{V\in T_xM\atop\left|h\left( V,V \right)\right|=1}h \left( V,\nabla f \left( x \right) \right)=\argmin _{V\in T_xM\atop\left|h\left( V,V \right)\right|=1}Vf \left( x \right).
\end{equation}
In fact, in both versions the search for the ``steepest descent direction'' is guided by making the directional derivative $Vf \left( x \right)$ as negative as possible, but constrained on different unit spheres. The precise relation between the two steepest descent directions is not readily visible, for the two unit spheres could differ drastically in geometry. In fact, for cases in which the unit ball $\left\{ v\in T_xM\mid \left| h \left( v,v \right) \right| = 1\right\}$ is noncompact, the ``steepest descent direction'' so defined may not even exist.

\begin{example}
 Consider the optimization problem over the Minkowski space $\mathbb{R}^{1,1}$ equipped with a metric of signature $\left( -, + \right)$
 \begin{equation*}
   \min_{x,y\in\mathbb{R}}f \left( x,y \right)=\frac{1}{2}\left[ x^2+\left( y+1 \right)^2 \right].
 \end{equation*}
At $\left( x,y \right)=\left( 0,0 \right)$, recall from \cref{exm:semi-riem-euc} that $\nabla f \left( 0,0 \right)=\left( 0,1 \right)^{\top}=Df \left( 0,0 \right)$. Over the unit ball $\left\{ \left( u,v \right)^{\top}\in\mathbb{R}^2\mid u^2-v^2=\pm1 \right\}\subset T_{\left( 0,0 \right)}\mathbb{R}^2$ under this Lorentzian metric, the scalar product $\left\langle Df \left( 0,0 \right), \left( u,v \right)^{\top} \right\rangle=v\rightarrow-\infty$ as $\left( u,v \right)\rightarrow \left( \infty, -\infty \right)$. Even worse, since the scalar product approaches $-\infty$, it is not possible to find a descent direction $\eta$ with $\left\langle Df \left( 0,0 \right) \right\rangle\geq \gamma \min _{V\in T_xM,\,\left|\,\left\langle V,V \right\rangle\right|=1}Vf \left( 0,0 \right)$ for some pre-set threshold $\gamma>0$.
\end{example}

One way to fix this non-compactness issue is to restrict the candidate tangent vectors $V$ in the minimization of $Vf \left( x \right)$ to lie in a compact subset of the tangent space $T\!M$. For instance, one can consider the unit sphere in $T\!M$ under a Riemannian metric. Comparing the right hand sides of \cref{eq:riem-steepest-descent} and \cref{eq:semi-riem-steepest-descent}, descent directions determined in this manner will be the negative gradient direction under the Riemannian metric, thus in general has nothing to do with the semi-Riemannian metric; moreover, if a Riemannian metric has to be defined laboriously in addition to the semi-Riemannian one, in principle we can already employ well-established, fully-functioning Riemannian optimization techniques, thus bypassing the semi-Riemannian setup entirely. While this argument might well render first-order semi-Riemannian optimization futile, we emphasize here that one can define steepest descent directions with the aid of ``Riemannian structures'' that arise naturally from the semi-Riemannian structure, and thus there is no need to specify a separate Riemannian structure in parallel to the semi-Riemannian one, though this affiliated ``Riemannian structure'' is highly local.

The key observation here is that one does not need to consistently specify a Riemannian structure over the entire manifold, if the only goal is to find one steepest descent direction in that tangent space --- in other words, when we search for the steepest descent direction in the tangent space $T_xM$ of a semi-Riemannian manifold $M$, it suffices to specify a Riemannian structure locally around $x$, or more extremely, only on the tangent space $T_xM$, in order for the ``steepest descent direction'' to be well-defined over a compact subset of $T_xM$. These local inner products do not have to ``patch together'' to give rise to a globally defined Riemannian structure. A very handy way to find local inner products is through the help of \emph{geodesic normal coordinates} that reduce the local calculation to the Minkowski spaces. For any $x\in M$, there is a normal neighborhood $U\subset M$ containing $x$ such that the exponential map $\mathrm{exp}_x:T_xM\rightarrow M$ is a diffeomorphism when restricted to $U$, and one can pick an orthonormal basis (with respect to the semi-Riemannian metric on $M$), denoted as $\left\{ e_1,\cdots,e_n \right\}$, such that $\left\langle e_i,e_j \right\rangle_x=\delta_{ij}\epsilon_j$, where $1\leq i,j\leq n$, $n=\mathrm{dim}\left( M \right)$, $\delta_{ij}$ are the Kronecker delta's, and $\epsilon_j\in \left\{ \pm 1 \right\}$. Without loss of generality, assume $M$ is a semi-Riemannian manifold of order $p$, where $0\leq p\leq n$, and that $\epsilon_1=\cdots=\epsilon_p=-1$, $\epsilon_{p+1}=\cdots=\epsilon_n=1$. The normal coordinates of any $y\in U$ are determined by the coefficients of $\exp_x^{-1}y\in T_xM$ with respect to the orthonormal basis $\left\{ e_1,\cdots,e_n \right\}$. It is straightforward (see \cite[Proposition 33]{ONeill1983}) to verify that

\begin{equation*}
  g_{ij}\left( x \right)=\delta_{ij}\epsilon_j, \quad \Gamma_{ij}^k \left( x \right)=0\qquad \forall 1\leq i,j,k \leq n
\end{equation*}
where $\left\{g_{ij}\right\}$ denotes the semi-Riemannian metric tensor components and $\left\{\Gamma_{ij}^k\right\}$ stands for the Christoffel symbols. Under this coordinate system, it is straightforward to verify that the scalar product between tangent vectors $u,v\in T_xM$ can be written as
\begin{equation*}
  \left\langle u,v \right\rangle=\sum_{i=1}^n\epsilon_iu^iv^i
\end{equation*}
where $u=u^ie_i$ and $v=v^je_j$ (Einstein's summation convention implicitly invoked). The local Riemannian structure can thus be defined as
\begin{equation}
  \label{eq:ptwise-riem-metric}
  g \left( u,v \right)=\sum_{i=1}^nu^iv^i.
\end{equation}
Essentially, such a local inner product is defined by imposing orthogonality between positive and negative definite subspaces of $T_xM$ and ``reversing the sign'' of the negative definite component of the scalar product. Making such a modification consistently and smoothly over the entire manifold is certainly subject to topological obstructions; nevertheless, locally (in fact, pointwise) defined Riemannian structures suffice for our purposes, and in practical applications we can simply the workflow by choosing an arbitrary orthonormal basis in the tangent space in place of the geodesic frame. The orthonormalization process, of course, is adapted for the semi-Riemannian setting; see \cite[Chapter 2, Lemma 24 and Lemma 25]{ONeill1983} or \cref{alg:semi-riem-basis-pursuit}. The output set of vectors $\left\{ e_1,\cdots,e_n \right\}$ satisfies
\begin{equation*}
  \left\langle e_i,e_j \right\rangle=\delta_{ij}\epsilon_i
\end{equation*}
where $\delta_{ij}$ are the Kronecker symbols, and $\epsilon_i=\left\langle e_i,e_i \right\rangle=\pm1$. A generic approach which works with high probability is to pick a random linearly independent set of vectors and apply a (pivoted) Gram-Schmidt orthogonalization process with respect to the indefinite scalar product; see \cref{alg:semi-riem-qr}.

\begin{algorithm}
\caption{{\sc Finding an Orthonormal Basis with respect to a Nondegenerate Indefinite Scalar Product}}
\label{alg:semi-riem-basis-pursuit}
\begin{algorithmic}[1]
  \Require{Vector space $V$ of finite dimension $n\in\mathbb{N}$, scalar product $\left\langle \cdot,\cdot \right\rangle:V\times V\rightarrow\mathbb{R}$ of type $\left( p,q \right)$ with $p+q=n$}
  \Function{FindONBasis}{$V$}
  \State Find $v\in V$ with $\left\langle v,v \right\rangle\neq 0$\quad \Comment{$v$ exists by nondegeneracy}
  \State $e_1\gets v/\sqrt{\left| \left\langle v,v \right\rangle \right|}$
  \For{$k=2,\cdots,n$}
  \State $V_k\gets \mathrm{span}\left\{ e_1,\cdots,e_{k-1} \right\}$
  \State $W_k\gets V_k^{\perp}$ \Comment{$V=V_k\oplus V_k^{\perp}$ by \cite[Lemma 3.19 and 3.23]{ONeill1983}}
  \State Find $w_k\in W_k$ with $\left\langle w_k,w_k \right\rangle\neq 0$\quad \Comment{$w_k$ exists by nondegeneracy of $W_k$}
  \State $e_k\gets w_k/\sqrt{\left| \left\langle w_k,w_k \right\rangle \right|}$
  \EndFor
  \State \Return $\left\{ e_1,\cdots,e_n \right\}$
  \EndFunction
 \end{algorithmic}
\end{algorithm}

\begin{algorithm}
\caption{{\sc Gram-Schmidt for an Indefinite Scalar Product}}
\label{alg:semi-riem-qr}
\begin{algorithmic}[1]
  \Require{Vector space $V$ of finite dimension $n\in\mathbb{N}$, scalar product $\left\langle \cdot,\cdot \right\rangle:V\times V\rightarrow\mathbb{R}$ of type $\left( p,q \right)$ with $p+q=n$, input linearly independent vectors $\left\{ v_1,\cdots,v_n \right\}$}
  \Function{IndefGramSchmidt}{$\left\{ v_1,\cdots,v_n \right\}$}
  \State $e_1\gets v_1/\sqrt{\left| \left\langle v_1,v_1 \right\rangle \right|}$ \Comment{w.l.o.g. assume $\left\langle v_1,v_1 \right\rangle\neq 0$}
  \For{$k=2,\cdots,n$}
  \State $\displaystyle w_k\gets v_k-\sum_{\ell=1}^{k-1}\left\langle v_k,v_{\ell} \right\rangle v_{\ell}$ \Comment{w.l.o.g. assume $\left\langle w_k,w_k \right\rangle\neq 0$ after pivoting}
  \State $e_k\gets w_k/\sqrt{\left| \left\langle w_k,w_k \right\rangle \right|}$
  \EndFor
  \State \Return $\left\{ e_1,\cdots,e_n \right\}$
  \EndFunction
 \end{algorithmic}
\end{algorithm}

In geodesic normal coordinates, the gradient $Df$ takes the form
\begin{equation*}
  Df \left( x \right)=\sum_{i=1}^n \epsilon_i\partial_if \left( x \right)\partial_i\big|_x \label{eq:semi-riem-gradient}
\end{equation*}
and choosing the steepest descent direction reduces to the problem
\begin{equation*}
  \max_{v^1,\cdots,v^n\in\mathbb{R}\atop \left( v^1 \right)^2+\cdots+\left( v^n \right)^2=1}\sum_{i=1}^n\epsilon_iv^i\partial_if \left( x \right)
\end{equation*}
of which the optimum is obviously attained at
\begin{equation*}
  \left( v^1,\cdots,v^n \right)=\frac{1}{\displaystyle \sum_{i=1}^n \left( \partial_if \left( x \right) \right)^2}\left( \epsilon_1\partial_1f \left( x \right), \cdots, \epsilon_n\partial_nf \left( x \right)\right).
\end{equation*}
For the simplicity of statement, we introduce the notation
\begin{equation*}
  \left[ X \right]^{+}:=\sum_{i=1}^n \left\langle X, e_i \right\rangle e_i
\end{equation*}
for $X\in T_xM$, where $\left\{e_1,\cdots,e_n\right\}$ is an orthonormal basis for the semi-Riemannian metric tensor $\left\langle \cdot,\cdot \right\rangle$ on $T_xM$. Using this notation, the descent direction we will choose can be written as
\begin{equation}
\label{eq:flip-negative}
  -\left[ Df \left( x \right) \right]^+=-\sum_{i=1}^n \left\langle Df \left( x \right), e_i \right\rangle e_i.
\end{equation}
Note that, by \cite[Lemma 3.25]{ONeill1983}, with respect to an orthonormal basis $\left\{ e_1,\cdots,e_n \right\}$ we have in general
\begin{equation*}
  Df \left( x \right)=\sum_{i=1}^n\epsilon_i \left\langle Df \left( x \right), e_i \right\rangle e_i\neq \sum_{i=1}^n \left\langle Df \left( x \right), e_i \right\rangle e_i=\left[ Df \left( x \right) \right]^+
\end{equation*}
which is consistent with our previous discussion that the steepest descent direction in the semi-Riemannian setting is not $-Df \left( x \right)$ in general. Intuitively, the ``steepest descent direction'' is obtained by reversing signs of components of the gradient that ``corresponds to'' the negative definite subspace, and then rescale according to the induced Riemannian metric. This leads to the routine \cref{alg:finding-steepest-descent-dir} for finding descent directions.

\begin{remark}
  The definition $\left[ X \right]^+$ certainly depends on the choice of the orthonormal basis with respect to the semi-Riemannian metric tensor. In other words, if we choose a different orthonormal basis with respect to the same semi-Riemannian metric on $T_xM$, the resulting descent direction will also be different. In practical computations, we could pre-compute an orthonormal basis for all points on the manifold, but that will complicate the proofs for convergence since the amount of descent will be uncomparable to each other across tangent vectors. A compromise is to cover the entire semi-Riemannian manifold with a chart consisting of geodesic normal neighborhoods, and extend the definition \cref{eq:flip-negative} from at a single point to over the geodesic normal neighborhood around each point, with the orthonormal basis given by \emph{geodesic normal frame fields} \cite[pp.84-85]{ONeill1983} defined over each normal neighborhood. Under suitable compactness assumptions, this construction essentially defines a Riemannian structure on the semi-Riemannian manifold by means of partition of unity and
  \begin{equation}
    \label{eq:induced-riem-structure}
    g \left( X,Y \right):=\left\langle X,\left[ Y \right]^+ \right\rangle=\sum_{i=1}^n \left\langle X,e_i \right\rangle \left\langle Y,e_i \right\rangle.
  \end{equation}
 The arbitrariness of the choice of geodesic normal frame fields makes this Riemannian structure non-canonical, but the bilinear form $g \left( \cdot,\cdot \right)$ is symmetric and coercive, and can thus be used for performing steepest descent in the semi-Riemannian setting.
\end{remark}

\begin{algorithm}
\caption{{\sc Finding Semi-Riemannian Descent Direction}}
\label{alg:finding-steepest-descent-dir}
\begin{algorithmic}[1]
  \Function{FindDescentDirection}{$x,M,Df \left( x \right)$}
  \State $\left\{ e_1,\cdots,e_n \right\}\gets${\sc FindONBasis}$\left( T_xM, \left\langle \cdot,\cdot \right\rangle \right)$
  \State $\displaystyle\eta\gets -\left[ Df \left( x \right) \right]^+=-\sum_{i=1}^n \left\langle Df \left( x \right), e_i \right\rangle e_i$
  \State \Return $\eta$
  \EndFunction
 \end{algorithmic}
\end{algorithm}

\begin{remark}
  For Minkowski spaces, it is easy to check that the descent direction output from \cref{alg:finding-steepest-descent-dir} coincides with $-\nabla f \left( x \right)$ exactly. In this sense \cref{alg:gd-semi-mfld} can be viewed as a generalization of the Riemannian steepest descent algorithm. In fact, the pointwise construction of positive-definite scalar products in each tangent space \cref{eq:ptwise-riem-metric} indicates that the methodology of Riemannian manifold optimization can be carried over to settings with weaker geometric assumptions, namely, when the inner product structure on the tangent spaces need not vary smoothly from point to point. From this perspective, we can also view semi-Riemannian optimization as a type of manifold optimization with weaker geometric assumptions.
\end{remark}

\begin{remark}
  \cref{alg:gd-semi-mfld} can indeed be viewed as an instance of a more general paradigm of line-search based optimization on manifolds \cite[\S3]{RW2012}. Our choice of the descent direction in \cref{alg:finding-steepest-descent-dir} ensures that the objective function value indeed decreases, at least for sufficiently small step size, which further facilitates convergence. %the establishment of convergence to stationary points. See \Cref{sec:some-global-local} for more details.
\end{remark}

\begin{example}[Semi-Riemannian Gradient Descent for Minkowski Spaces]
  \label{exm:semi-riem-grad-descent-optim}
  Recall from \cref{exm:semi-riem-euc} that the semi-Riemannian gradient of a differentiable function on Minkowski space $\mathbb{R}^{p,q}$ is $Df \left( x \right)=\I_{p,q}\nabla f \left( x \right)$. If we choose the standard canonical basis for $\mathbb{R}^{p,q}$, the descent direction $\left[ Df \left( x \right) \right]^+$ produced by \cref{alg:finding-steepest-descent-dir} and needed for \cref{alg:gd-semi-mfld} is
  \begin{equation*}
    \left[ Df \left( x \right) \right]^+=\I_n\cdot\I_{p,q}\cdot\I_n\cdot\I_{p,q}\nabla f \left( x \right)=\nabla f \left( x \right)
  \end{equation*}
and thus the semi-Riemannian gradient descent coincides with the standard gradient descent algorithm on the Euclidean space if the standard orthonormal basis is used at every point of $\mathbb{R}^{p,q}$. Of course, if we use a randomly generated orthonormal basis (under the semi-Riemannian metric) at each point, the semi-Riemannian gradient descent will be drastically different from standard gradient descent on Euclidean spaces; see \Cref{sec:minkowski-spaces} for an illustration.
\end{example}

When studying self-concordant barrier functions for interior point methods, a useful guiding principle is to consider the Riemannian geometry defined by the Hessian of a strictly convex self-concordant barrier function \cite{NN1994,Duistermaat2001,Renegar2001,NT2002}; in this setting, descent directions produced from Newton's method can be equivalently viewed as gradients with respect to the Riemannian structure. When the barrier function is non-convex, however, the Hessians are no longer positive definite, and the Riemannain geometry is replaced with semi-Riemannian geometry. It is well known that the direction computed from Newton's equation \cref{eq:newton-equation} may not always be a descent direction if $\nabla^2f$ is not positive definite \cite[\textsection 3.3]{WN1999}, which is consistent with our observation in this subsection that semi-Riemannian gradients need not be descent directions in general. In this particular case, our modification \cref{eq:flip-negative} can also be interpreted as a novel variant of the \emph{Hessian modification} strategy \cite[\textsection 3.4]{WN1999}, as follows. Denote the function under consideration as $f:Q\rightarrow\mathbb{R}$, where $Q\subset \mathbb{R}^n$ is a connected, closed convex subset with non-empty interior and contains no straight lines. Assume $\nabla^2f$ is non-degenerate on $Q$, which necessarily implies that $\nabla^2f$ is of constant signature on $Q$. At any $x\in Q$, the negative gradient of $f$ with respect to the semi-Riemannian metric defined by the Hessian of $f$ is $-Df \left( x \right)=-\left[\nabla^2 f\left( x \right)\right]^{-1}\nabla f \left( x \right)$, where $\nabla f$ and $\nabla^2f$ stand for the gradient and Hessian of $f$ with respect to the Euclidean geometry of $Q$. Our proposed modification first finds a matrix $U\in\mathbb{R}^{n\times n}$ satisfying
\begin{equation*}
  U^{\top}\left[ \nabla^2f \left( x \right) \right]U=\I_{p,q}
\end{equation*}
where $\left( p,q \right)$ is the constant signature of $\nabla^2f$ on $Q$, and then set
\begin{equation}
\label{eq:evident}
  -\left[ Df \left( x \right) \right]^+=-UU^{\top}\left[ \nabla^2f \left( x \right) \right]Df \left( x \right)=-UU^{\top}\nabla f \left( x \right)
\end{equation}
which is guaranteed to be a descent direction since
\begin{equation*}
  -\left[\nabla f \left( x \right)\right]^{\top}\left[ Df \left( x \right) \right]^+=-\left\| U\nabla f \left( x \right) \right\|^2\leq 0.
\end{equation*}
From \cref{eq:evident} it is evident that the semi-Riemannian descent direction $-\left[ Df \left( x \right) \right]^{+}$ is obtained from $-Df \left( x \right)$ by replacing the inverse Hessian with $UU^{\top}$. This is close to Hessian modification in spirit, but also drastically different from common Hessian modification techniques that adds a correction matrix to the true Hessian $\nabla^2f \left( x \right)$; see \cite[\textsection 3.4]{WN1999} for more detailed explanation.

\subsection{Semi-Riemannian Conjugate Gradient}
\label{sec:conjugate-gradient}

Using the same steepest descent directions and line search strategy, we can also adapt conjugate gradient methods to the semi-Riemannian setting. See \cref{alg:cg-semi-mfld} for the algorithm description. Note that in \cref{alg:cg-semi-mfld} we used the Polak-Rebi{\`e}re formula to determine $\beta_k$, but alternatives such as Hestenes-Stiefel or Fletcher-Reeves methods (see e.g. \cite[\S2.6]{EAS1998} or \cite{RW2012}) can be easily adapted to the semi-Riemannian setting as well, since none of the major steps in Riemannian conjugate gradient algorithm relies essentially on the positive-definiteness of the metric tensor, except that the (steepest) descent direction needs to be modified according to \cref{eq:flip-negative}. We noticed in practice that Polak-Rebi{\`e}re and formulae tend to be more robust and efficient than the Fletcher-Reeves formula for the choice of $\beta_k$, which is consistent with general observations of nonlinear conjugate gradient methods \cite[\textsection 5.2]{WN1999}.

\begin{algorithm}
\caption{{\sc Semi-Riemannian Conjugate Gradient (Polak-Rebi{\`e}re)}}
\label{alg:cg-semi-mfld}
\begin{algorithmic}[1]
  \Require{Manifold $M$, objective function $f$, retraction $\mathrm{Retr}$, parallel transport $P$, initial value $x_0\in M$, parameters for {\sc Linesearch}, gradient $Df$ and Hessian $D^2f$}
  \State $k\gets 0$
  \State $x_0\gets$ {\sc Initiate}
  \State $\eta_0\gets${\sc FindDescentDirection}$\left( x_0, M, Df \left( x_0 \right) \right)$ \Comment{c.f. \cref{alg:finding-steepest-descent-dir}}
  \While{not converge}
  \State $0<t_k\gets$ {\sc LineSearch}$\left( f, x_k, \eta_k \right)$ \Comment{$t_k$ is the Armijo step size}
  \State $x_{k+1}\gets \mathrm{Retr}_{x_k}\left( t_k\eta_k \right)$
  \State $\xi_{k+1}\gets${\sc FindDescentDirection}$\left( x_{k+1}, M, Df \left( x_{k+1} \right) \right)$
  \State $\eta_{k+1}=\xi_{k+1}+\beta_kP\eta_k$, where \Comment{$P:T_{x_k}\!M\rightarrow T_{x_{k+1}}\!M$}
  \begin{equation*}
    \beta_k:=\max \left\{ 0, \frac{\left\langle Df \left( x_{k+1} \right)-P \left[ Df \left( x_k \right) \right], \left[ Df \left( x_{k+1} \right) \right]^+ \right\rangle} {\left\langle Df \left( x_k \right), \left[ Df \left( x_k \right) \right]^+ \right\rangle}\right\}%-\left[D^2f \left( x_{k+1} \right)\right]\left( \xi_{k+1},P\eta_k \right)/\left[D^2f \left( x_{k+1} \right)\right]\left( P\eta_k,P\eta_k \right)
  \end{equation*}
  \State $k\gets k+1$
  \EndWhile
  \State \Return Sequence of iterates $\left\{ x_k \right\}$
 \end{algorithmic}
\end{algorithm}

\begin{remark}
  For Minkowski spaces (including Lorentzian spaces) with the standard orthonormal basis, both steepest descent and conjugate gradient methods coincide with their counterparts on standard Euclidean spaces, since they share identical descent directions, parallel-transports, and Hessians of the objective function. 
\end{remark}

\begin{remark}
  \cref{alg:cg-semi-mfld} can also be applied to self-concordant barrier functions for interior point methods, when the objective function is not necessarily strictly convex but has non-degenerate Hessians. In this context, where the semi-Riemannian metric tensor is given by the Hessian of the objective function, \cref{alg:cg-semi-mfld} can be viewed as a hybrid of Newton and conjugate gradient methods, in the sense that the ``steepest descent directions'' are determined by the Newton equations but the actual descent directions are combined using the methodology of conjugate gradient methods. To the best of our knowledge, such a hybrid algorithm has not been investigated in existing literature.
\end{remark}

\subsection{Metric Independence of Second Order Methods}
\label{sec:first-and-second-order-algor}

In this subsection we consider two prototypical second-order optimization methods on semi-Riemannian manifolds, namely, Newton's method and trust region method. Surprisingly, both methods turn out to produce descent directions that are independent of the choice of scalar products on tangent spaces. We give a geometric interpretation of this independence from the perspective of \emph{jets} in \Cref{sec:trust-region-method}.

\subsubsection{Semi-Riemannian Newton's Method}
\label{sec:newtons-method}

As an archetypal second-order method, Newton's method on Riemannian manifolds has already been developed in detail in the early literature of Riemannian optimization \cite[Chap 6]{AMS2009}. The rationale behind Newton's method is that the first order stationary points of a differentiable function $f:M\rightarrow\mathbb{R}$ are in one-to-one correspondence with the minimum of $\left\| \nabla f \right\|^2=\left\langle \nabla f, \nabla f \right\rangle$ when the metric is positive-definite (i.e., when $M$ is a Riemannian manifold). Thus by choosing the direction $V$ to satisfy the Newton equation $\nabla_{V}\nabla f=-\nabla f$ we ensure that $V$ is a descent direction
\begin{equation*}
  V \left\langle \nabla f, \nabla f \right\rangle = 2 \left\langle \nabla_V \nabla f, \nabla f \right\rangle=-2 \left\langle \nabla f, \nabla f \right\rangle = -2 \left\| \nabla f \right\|^2
\end{equation*}
and the right hand side is strictly negative as long as $\nabla f\neq 0$. The main difficulty in generalizing this procedure to the semi-Riemannian setting is similar with the difficulty we faced in \Cref{sec:algorithms}: when the metric is indefinite, $Df=0$ has nothing to do with $\left\| Df \right\|=0$, and thus one can no longer find the stationary points of $f$ by minimizing $\left\| Df \right\|^2$. The approach we'll adopt to fix this issue is also similar to that in \Cref{sec:algorithms}: instead of minimizing $\left\langle Df \left( x \right),Df \left( x \right) \right\rangle$, we will focus on the coercive bilinear form $\left\langle Df \left( x \right), \left[ Df \left( x \right) \right]^+ \right\rangle$.

Let $E_1,\cdots,E_n$ be a local geodesic normal coordinate frame centered at $x\in M$, i.e. for any $1\leq i,j\leq n$
\begin{equation*}
  \left\langle E_i \left( x \right),E_j \left( x \right)\right\rangle=\epsilon_i\delta_{ij},\quad \nabla_{E_i}E_j \left( x \right)=0.
\end{equation*}
Then we have
\begin{equation}
  \label{eq:semi-riem-descent-surrogate}
  \left\langle Df \left( x \right), \left[ Df \left( x \right) \right]^+ \right\rangle=\sum_{i=1}^n \left| \left\langle Df \left( x \right), E_i \left( x \right) \right\rangle  \right|^2
\end{equation}
and thus for any tangent vector $V\in T_xM$ we have
\begin{equation*}
  \begin{aligned}
    V &\left\langle Df \left( x \right), \left[ Df \left( x \right) \right]^+ \right\rangle = 2\sum_{i=1}^n \left\langle Df \left( x \right), E_i \left( x \right) \right\rangle V \left\langle Df \left( x \right), E_i \left( x \right) \right\rangle\\
    &= 2\sum_{i=1}^n \left\langle Df \left( x \right),E_i \left( x \right)\right\rangle \left[ \left\langle D_VDf \left( x \right), E_i \left( x \right) \right\rangle+\left\langle Df \left( x \right), D_VE_i \left( x \right) \right\rangle\right]\\
    &=2\sum_{i=1}^n \left\langle Df \left( x \right),E_i \left( x \right)\right\rangle \left\langle D_VDf \left( x \right), E_i \left( x \right) \right\rangle
  \end{aligned}
\end{equation*}
where in the last equality we used the fact that $\nabla_{E_i}E_j \left( x \right)=0$ for all $1\leq i,j\leq n$. Therefore, as long as we pick $V$ to satisfy Newton's equation
\begin{equation}
  \label{eq:semi-riem-newton-equation}
  \left[D^2f \left( x \right)\right] \left( V \right) = D_VDf \left( x \right)=-Df \left( x \right)
\end{equation}
we can ensure decrease in the value of \cref{eq:semi-riem-descent-surrogate}. In other words, we can obtain a descent direction for semi-Riemannian optimization using the same Newton's equation as for Riemannian optimization, with the only difference that Riemannian gradient and Hessian get replaced with their semi-Riemannian counterparts.

\begin{algorithm}
\caption{{\sc Semi-Riemannian Newton's Method}}
\label{alg:newton-semi-mfld}
\begin{algorithmic}[1]
  \Require{Manifold $M$, objective function $f$, retraction $\mathrm{Retr}$, initial value $x_0\in M$, parameters for {\sc Linesearch}, gradient $Df$ and Hessian $D^2f$}
  \While{not converge}
  \State Obtain the descent direction by solving the Newton equation $$\left[D^2f \left( x_k \right)\right]\left(\eta_k\right)=-Df \left( x_k \right)$$
  \State $0<t_k\gets$ {\sc LineSearch}$\left( f, x_k, \eta_k \right)$ \Comment{$t_k$ is the Armijo step size}
  \State $x_{k+1}\gets \mathrm{Retr}_{x_k} \left( t_k\eta_k \right)$
  \State $k\gets k+1$
  \EndWhile
  \State \Return Sequence of iterates $\left\{ x_k \right\}$
 \end{algorithmic}
\end{algorithm}

Given that our semi-Riemannian Newton's method builds upon the ``Riemannian surrogate'' \cref{eq:semi-riem-descent-surrogate}, it is not surprising that the semi-Riemannian Newton's method reduces to the ordinary Newton's method on Minkowski spaces, and the geodesics and parallel-transports stays the same as their Riemannian counterparts (i.e. when the scalar product is positive definite). This is best illustrated in the following calculation.

\begin{example}[Semi-Riemannian Newton's Method for Minkowski Spaces]
  \label{exm:semi-riem-newton-optim}
  Recalling the definitions of semi-Riemannian gradient and Hessians from \cref{exm:semi-riem-euc}, the descent direction $\eta_k$ needed in \cref{alg:newton-semi-mfld} is determined by
  \begin{equation*}
    \I_{p,q}\nabla^2f \left( x_k \right)\eta_k=-\I_{p,q}\nabla f \left( x_k \right)\quad\Leftrightarrow\quad\eta_k=-\left[ \nabla^2f \left( x_k \right) \right]^{-1}\nabla f \left( x_k \right)
  \end{equation*}
  for all $k=0,1,2,\cdots$. This calculation made it clear that the semi-Riemannian Newton's method coincides with the standard Newton's method.
\end{example}

The metric independence demonstrated in \cref{exm:semi-riem-newton-optim} reflects a more general phenomenon of metric independence in Newton's method as formulated in  \cite[\textsection 1.6]{Renegar2001}. Though the discussion in phenomenon of metric independence in Newton's method as formulated in  \cite[\textsection 1.6]{Renegar2001} is restricted to the Riemannian case (scalar product required to be positive definite), it is straightforward to see that the metric independence persists under non-degenerate change of semi-Riemannian structures. In fact, if we denote $J \left( x_k \right)$ for the Jacobian matrix of a non-degenerate coordinate transformation at $x_k$, it is straightforward to check from the coordinate expressions of semi-Riemannian gradient and Hessian \cref{eq:riem-semiriem-grad-hess} that the Newton equation \cref{eq:semi-riem-newton-equation} in the new coordinate system takes the form $J \left( x_k \right)\left[ D^2f \left( x_k \right) \right]\left( V \right)=-J \left( x_k \right)\nabla f \left( x_k \right)$, which yields the same descent direction as \cref{eq:semi-riem-newton-equation}. In the Riemannian regime, this metric independence is often attributed to the fact that second-order approximation is independent of inner products (see e.g. \cite[\textsection 1.6]{Renegar2001}); we provide a general and unified differential geometric interpretation of this independence in terms of \emph{jets} in \Cref{sec:trust-region-method}.

%\TG{resume from here: change of coordinate perspective; link to jets below.}

% though the context there is restricted to the Ricase when the scalar product is positive definite, the same argument carries through straightforwardly for non-degenerate scalar products.. 

\subsubsection{Jets and the Metric Independence of Trust Region Method}
\label{sec:trust-region-method}

It is well known that first-order and Newton's methods suffer from various drawbacks from a numerical optimization methods, such as slow local convergence and/or prohibitive computational cost in determining the descent direction. It is thus argued (c.f. \cite{AMS2009}, \cite{ABG2007}) that it  could be more efficient to consider successive optimization of local models of the cost function on the domain of the problem. \emph{Trust region methods}, which considers quadratic local models through approximate Taylor expansions of the cost function, fall into this category (see e.g. \cite{WN1999} and the references therein). This methodology has also been generalized to Riemannian manifolds for manifold optimization \cite{ABG2007,AMS2009,HAG2015,HS2017}. In a nutshell, at each point $x\in M$ the Riemannian trust-region method strives to find the descent direction by solving locally the quadratic optimization problem on the tangent plane $T_xM$:
\begin{equation}
\label{eq:local-model-riem}
  \min_{\eta\in T_xM\atop \left\| \eta \right\|\leq \Delta_0}m_x \left( \eta \right)=f \left( x \right)+\left\langle \nabla f \left( x \right), \eta \right\rangle+\frac{1}{2}\left[\nabla^2f \left( x \right)\right]\left( \eta, \eta \right)
\end{equation}
where $\left\langle \cdot,\cdot \right\rangle$ is the inner product specified by the Riemannian metric tensor, $\left\| \cdot \right\|$ is the induced norm, and $\Delta_0$ is the radius of the trust region which is updated through the iterations according to certain technical criteria (e.g. the geometry of the manifold, the approximation quality of the local model, etc.).

When generalizing trust region methods to semi-Riemannian optimization, again we are faced with the difficulties for the other methods discussed previously, such as the non-compactness of the ``metric ball'' of bounded radius $\Delta_0>0$. This can be resolved by introducing a positive definite inner product accompanying the indefinite metric tensor as in \Cref{sec:algorithms} and \Cref{sec:newtons-method}, then restrict the search for the descent direction to a bounded domain defined by the norm induced from the inner product. Denoting $\left\| \cdot \right\|_+$ for the induced norm on $T_xM$, the local quadratic optimization problem in the semi-Riemannian setting can be written as
\begin{equation}
\label{eq:local-model-semi-riem}
  \min_{\eta\in T_xM\atop \left\| \eta \right\|_+ \leq \Delta_0}m^{\textrm{semi}}_x \left( \eta \right)=f \left( x \right)+\left\langle D f \left( x \right), \eta \right\rangle+\frac{1}{2}\left[D^2f \left( x \right)\right]\left( \eta, \eta \right).
\end{equation}
We argue that this local quadratic model coincides with the Riemannian model \cref{eq:local-model-riem} with the (frame-field-dependent) Riemannian structure \cref{eq:induced-riem-structure}. In fact, the verification is straightforward by picking geodesic normal coordinate systems under the Riemannian and semi-Riemannian metric (which ensures the Christoffel symbols vanish at $x$) and a change-of-coordinate argument as in the proof of \cref{prop:equivalence-geo-convexity}, together with the coordinate expressions \cref{eq:riem-semiriem-grad-hess}. This implies that a trust region method based on \cref{eq:local-model-semi-riem} for the semi-Riemannian manifold $M$ can be interpreted and analyzed using more or less the same techniques in existing literature of Riemannian trust region methods. The only subtlety here is the frame dependence of locality of the Riemannian structures accompanying the semi-Riemannian metric; nevertheless, this technicality can be resolved by noticing the direct dependence of the local Riemannian structure with the smooth semi-Riemannian structure.

The argument we gave in this section can be carried out to establish the ``metric independence'' of trust region methods on manifolds. While it is certainly desirable to pick a metric on the manifold so as to enable numerical implementations of the optimization algorithms, at the end of the day the only influence of the metric enters the trust region methods through choosing the size $\Delta_0$ of the trust region, which eventually does not matter after the region radius update rules are carried out (which ultimately depends on the value distribution of the cost function only). One geometric explanation for this phenomenon is through the notion of \emph{jets} (see e.g. \cite{Saunders1989,Vakil1998,Palais2016}), which characterizes the manifold analogy of ``polynomial approximation'' for smooth functions. Though the formal invarance of under change of coordinates breaks down for derivatives greater than or equal to the second order, it turns out that one can define equivalence classes of ``Taylor polynomial expansion modulo higher order terms'' by the matching of a fixed number of lower order derivatives at a fixed point. More concretely, consider an arbitrary point $q\in M$ and denote $\left(U, \left(x^1,\cdots,x^d\right)\right)$ for a coordinate system around $q$, and assume without loss of generality that $x^j\left(q\right)=0$ for all $j=1,\cdots,d$. By a direct calculate, one can verify that the second order Taylor expansion
\begin{equation}
\label{eq:taylor-expansion-coord}
f \left(x\right) = f\left(0\right)+x^i\partial_if \left(0\right)+\frac{1}{2}x^jx^k\partial^2_{jk}f\left(0\right)+O\left(\left\|x\right\|^3\right),\quad x\in U
\end{equation}
is formally preserved under change of coordinates up to cubic polynomials. This indicates that, as long as we interpret the big-$O$ notation in \cref{eq:taylor-expansion-coord} as containing not only ``metrically'' $O \left( \left\| x \right\|^3 \right)$ terms (characterized by the local smooth structure or the metric tensor thereof) but also polynomials of degree $\geq 3$ in the components of $x\in\mathbb{R}^d$, then the expansion \cref{eq:taylor-expansion-coord} makes sense geometrically as an element in the polynomial ring modulo ideals generated by cubic polynomials. (In fact, for fixed $k\in\mathbb{N}$, the union of $k$-jets over all points on the manifold form a fibre bundle often referred to as a \emph{jet bundle}.) For the purpose of trust region methods this equivalence relation suffices for specifying local models, as equivalent polynomials (as the same jet) give rise to local models of the same \emph{order} (see e.g. \cite[Proposition 7.1.3]{AMS2009}). It then follows that, for distinct Riemannian or semi-Riemannian metrics on the same smooth manifold and under geodesic normal coordinates chosen respectively with respect to the metric structures, the local models \cref{eq:local-model-riem,eq:local-model-semi-riem} correspond to the same jet and will metrically differ from each other in terms of cubic geodesic distances only, whenever the metrics involved are all Riemannian. When at least one of the metric tensors involved is semi-Riemannian, the metric comparison has to be carried out with extra caution (e.g. with respect to the metric structure induced by another Riemannian structure) since coordinate polynomials are no longer bounded by ``semi-Riemannian norms'' of the same order, again due to the indefiniteness of the semi-Riemannian metric tensor.

% \subsection{Some Convergence Analysis}
% \label{sec:some-global-local}

% \TG{Do we need proof for convergence? These will be lengthy but essentially identical with Riemannian proofs. These are not even ``literally translations'' of the Riemannian proofs, but ``exactly'' the Riemannian proofs themselves --- because our approaches are essentially ``Riemannianizing'' the semi-Riemannian setting in hand.}

\section{Semi-Riemannian Optimization on Submanifolds}
\label{sec:semi-riem-geom-sub}

Submanifolds of Euclidean spaces are most often encountered in practical applications of manifold optimization. A key difference between Riemannian and semi-Riemannian geometry is that the non-degeneracy of the metric tensor can not be inherited by sub-manifolds as easily from semi-Riemannian ambient manifolds: for a submanifold $X$ of $M$, any Riemannian metric on $M$ induces a Rimannian metric on $X$ since $g$ is positive definite at every point $x\in X$, but a semi-Rimannian metric on $M$ could become degenerate when restrict to $X$; this degeneracy is the main obstruction to finding a well-defined ``orthogonal projection'' which is essential for (i) relating gradients on the manifold with gradients in the ambient space, and (ii) defining geodesics on submanifolds. Semi-Riemannian manifolds with degeneracy are of interest to the theory of general relativity and mathematical physics; see \cite{Kupeli1987a,Kupeli1987b,Kupeli1987c,Larsen1992,Stoica2014} and the references therein. This section provides some characterization of degenerate semi-Riemannian manifolds (see \cref{defn:degenerate manifold}) in terms of their \emph{degenerate bundles} (see \cref{def:degenerate bundle}). The goal is to identify non-degenerate semi-Riemannian submanifolds of Minkowski spaces for which our algorithmic framework in \Cref{sec:semi-riem-optim} applies. As demonstrated in the computation in this section and \Cref{sec:additional-examples}, unfortunately, semi-Riemannian structures inherited from the ambient Minkowski space are degenerate for most matrix Lie groups. Nonetheless, many interesting hypersurfaces (co-dimension one submanifolds) of Minkowski spaces admit non-degenerate induced semi-Riemannian structures, or degenerate ones but with degeneracy contained in a set of measure zero; the semi-Riemannian optimization  framework introduced in \Cref{sec:semi-riem-optim} applies seamlessly to these examples, some of which we illustrate in \Cref{sec:numer-exper}.

%, which makes $X$ a degenerate semi-Riemannian manifold (see Definition \ref{defn:degenerate manifold}). The ``degenerate component'' of a semi-Riemannian manifold is often characterized by its \emph{degenerate bundle}; see Definition \ref{def:degenerate bundle} below. This section is devoted to first steps towards coping with such degeneracy in practical optimization applications.
 
\subsection{Degeneracy of Semi-Riemannian Submanifolds}
\label{sec:degen-subm}

Theories of Riemannian and semi-Riemannian geometry build upon the non-degeneracy of metric tensors. However, physical models of spacetime renders itself naturally to the occurrence of \emph{singularities}, as pointed out in general relativity \cite{Penrose1965,Hawking1966,HawkingPenrose1970,HawkingEllis1973}. A lot of work in semi-Riemannian geometry are thus devoted to the development of \emph{singular semi-Riemannian geometry} --- the geometry of semi-Riemannian manifolds with degeneracy in their metric tensors, either with constant signature \cite{Kupeli1987a,Kupeli1987b,Kupeli1987c} or more generally, with possibly variable signature \cite{Larsen1992,Stoica2014}. In special cases such as null hypersurfaces of Lorentzian manifolds, specific techniques such as \emph{rigging} \cite{GO2016,AGH2018} have been developed, but generalizing these special constructions to other degenerate semi-Riemannian submanifolds is much less straightforward, if possible at all. For the simplicity of exposition, we'll confine our discussion to the constant signature scenario regardless of whether singularities occur.

\begin{definition}
  \label{defn:vector-types}
  A symmetric bilinear form $\left\langle\cdot,\cdot \right\rangle:V\times V\rightarrow \mathbb{R}$ on a vector space $V$ is said to have \emph{signature} $\left( \kappa, \nu, \pi \right)$ if the maximum positive definite subspace is of dimension $\pi\in\mathbb{Z}_{\geq 0}$, the maximum negative definite subspace is of dimension $\nu\in\mathbb{Z}_{\geq 0}$, and the dimension of the \emph{degenerate subspace} with respect to this bilinear form
  \begin{equation*}
    V^{\perp} := \left\{ v\in V\mid \left\langle v,u \right\rangle=0\quad\forall u\in V\right\}
  \end{equation*}
  is of dimension $\kappa\in\mathbb{Z}_{\geq 0}$. A vector $0\neq v\in V$ is said to be (1) \emph{degenerate} if $v\in V$; (2) \emph{null} if $\left\langle v,v \right\rangle=0$ but $v\notin V^{\perp}$; (3) \emph{timelike} if $\left\langle v,v \right\rangle<0$; (4) \emph{spacelike} if $\left\langle v,v \right\rangle>0$.
\end{definition}
\begin{definition}
  Let $W$ be a subspace of a vector space $V$ equipped with a bilinear form $\left\langle \cdot,\cdot \right\rangle:V\times V\rightarrow\mathbb{R}$. Denote $\left( \kappa, \nu, \pi \right)$ for the type of the bilinear form on $W$ obtained from restricting $\left\langle \cdot,\cdot \right\rangle$ to $W$. We say that $W$ is (1) \emph{degenerate} if $\kappa \geq 1$; (2) \emph{nondegenerate} if $\kappa=0$; (3) \emph{timelike} if $\kappa=0$ and $\nu\geq 1$; (4) \emph{spacelike} if $\kappa=0$, $\nu=0$, and $\pi\geq 1$.
\end{definition}

\begin{definition}[Degenerate Semi-Riemannian Manifolds]\label{defn:degenerate manifold}
  A \emph{degenerate semi-Riemannian manifold} is a smooth manifold equipped with a possibly degenerate $\left( 0,2 \right)$ tensor field. This tensor field will be referred to as the \emph{degenerate metric tensor} of the degenerate semi-Riemannian manifold; the signature of the degenerate metric tensor will also be referred to as the signature of the manifold when no confusion exists. Unless otherwise specified, the degenerate metric tensor is of constant signature in the rest of this paper.
\end{definition}

When the context is clear, we will occasionally omit the adjective ``degenerate'' when referring to degenerate semi-Riemannian manifolds and the degenerate metric tensor on it, since non-degenerate semi-Riemannian manifolds are special cases of degenerate ones with $\kappa=0$.

\begin{definition}[Degenerate Bundle, \cite{Kupeli1987c} Definition 3.1]\label{def:degenerate bundle}
  The \emph{degenerate bundle} of a (possibly degenerate) semi-Riemannian manifold $\left( M,\left\langle \cdot,\cdot \right\rangle \right)$ is defined as the distribution
  \begin{equation}
    \label{eq:degen-bundle}
    M^{\perp}:=\bigcup_{x\in M} \left\{ u\in T_xM\mid \left\langle u,v \right\rangle=0\quad\forall v\in T_xM\right\}.
  \end{equation}
  We say $M$ is \emph{integrable} if the distribution $M^{\perp}$ is integrable. We denote by $M^{\perp}_x$ the linear space $\{u\in \T_x M: \langle u,v \rangle =0,~\forall v\in \T_x M\}$ and we call the set of point $x\in M$ such that $M_x^{\perp}\ne \{0\}$ the \emph{degenerate locus} of $M$.
\end{definition}

As in the setup of Riemannian manifold optimization, for practice it is of primary interest to understand the differential geometry of submanifolds of an ambient manifold for which most differential geometric quantities can be characterized explicitly. In the context of semi-Riemannian geometry, a first technical subtlety with the notion of ``semi-Riemannian submanifolds'' is that the induced semi-Riemannian metric tensor may well suffer from certain degeneracy even when the ambient semi-Riemannian geometry is non-degenerate. The main difficulty lies at the non-existence of a canonical ``orthogonal projection'' from the ambient to the submanifold tangent spaces --- this complicates the definitions of normal bundles, second fundamental forms, as well as extrinsic characterizations of intrinsic geometric concepts such as affine connections, geodesics, and parallel-translates. For instance, it is well-known that covariant derivatives on a semi-Riemannian submanifold can be obtained from calculating the covariant derivatives on the ambient semi-Riemannian manifold and then projecting the result to the tangent spaces of the submanifold (see e.g. \cite[Chapter 4, Lemma 3]{ONeill1983}), but this characterization breaks down if the projection operator can not be properly defined. In fact, on a degenerate semi-Riemannian manifold there does not exist in general a semi-Riemannian analogue of the Levi-Civita (metric-compatible and torsion-free) connection, even for a degenerate semi-Riemannian submanifold of a non-degenerate semi-Riemannian manifold. Such an analogue, if exists, is called a \emph{Koszul derivative} of the degenerate semi-Riemannian manifold; a semi-Riemannian manifold admitting a Koszul derivative is called a \emph{singular semi-Riemannian manifold} in \cite{Kupeli1987a,Kupeli1987b,Kupeli1987c}. In general, a singular semi-Riemannian manifold $M$ admits more than one Koszul derivatives, and any two Koszul derivatives on $M$ differ from each other by a map from $\Gamma \left( TM \right)\times \Gamma \left( TM \right)$ to the degenerate bundle $M^{\perp}$; see e.g. \cite[Proposition 3.5]{Kupeli1987a}. Note that though it is tempting to define a connection on a degenerate semi-Riemannian manifold through the Koszul formula \cref{eqn:Koszul formula}, the formula defines a Koszul derivative if and only if the metric tensor is Lie parallel along all sections of the degenerate bundle (\cite[Theorem 3.4]{Kupeli1987c}). Another useful (necessary but insufficient) criterion for the existence of a Koszul derivative on a degenerate semi-Riemannian manifold is the integrability the degenerate bundle: as shown in \cite[Corollary 3.6]{Kupeli1987c}, if a semi-Riemannian manifold $M$ admits a Koszul derivative, then $M^{\perp}$ is integrable.

A large class of examples of semi-Riemannian manifolds commonly encountered in scientific computation are matrix Lie groups. They admit semi-Riemannian structures of arbitrary signature since tangent bundles of Lie groups are trivial. For instance, it is straightforward to verify that the semi-Riemannian structure on $\mathbb{R}^{n\times n}$ specified in \cref{exm:matrix-semi-riem} induces a non-degenerate semi-Riemannian structure on the general linear group $\GL \left( n,\mathbb{R} \right)$, though non-degeneracy becomes evident for almost all interesting matrix subgroups of $\GL \left( n,\mathbb{R} \right)$. We demonstrate the ubiquity of such degeneracy in the following two examples; more examples of matrix Lie groups are deferred to \Cref{sec:semi-riem-geom-matrix-lie-groups}. %and propose an agenda to reconcile with this deficiency in \Cref{sec:embedded-geom}.

\begin{example}[Indefinite Orthogonal Group]
  \label{exm:indef-orth-group}
  Let $\mathbb{N}\ni n=p+q$, $0\leq p\leq n$, and $p,q\in \mathbb{N}$. Define the indefinite orthogonal group of signature $\left( p,q \right)$ as
  \begin{equation}
    \label{eq:defn-indefinite-orth}
    O \left( p,q \right):=\left\{ A\in\mathbb{R}^{n\times n}\mid A^{\top}\I_{p,q}A=\I_{p,q} \right\}
  \end{equation}
  where $\I_{p,q}$ is defined in \cref{exm:matrix-semi-riem}. The Lie algebra of this Lie group can be easily verified as
  \begin{equation}
    \label{eq:defn-indefinite-sitefel-lie-algebra}
    \mathfrak{o}(p,q) \coloneqq \left\lbrace
X\in \mathbb{R}^{n \times n}: X^{\top} \I_{p,q} + \I_{p,q} X = 0
\right\rbrace.
\end{equation}
The tangent space at an arbitrary $A\in O \left( p,q \right)$ is thus
\begin{equation}
  \label{eq:indefinite-orth-tangent-space}
  \begin{aligned}
    T_A O \left( p,q \right)&=\left\{  AX\mid X\in\mathbb{R}^{n\times n}, X^{\top}\I_{p,q}+\I_{p,q}X=0 \right\}\\
    &=\left\{ A\I_{p,q}Y\mid Y\in\mathbb{R}^{n\times n}, Y^{\top}+Y=0 \right\}.
  \end{aligned}
\end{equation}
Equipping $\mathfrak{o}\left( p,q \right)$ with bilinear form specified in \cref{exm:matrix-semi-riem}, the Lie group structure on $O \left( p,q \right)$ induces a left-invariant semi-Riemannian metric on $O \left( p,q \right)$ by
\begin{equation}
  \label{eq:def-indef-orth-natural-semi-riem-metri}
  \begin{aligned}
  \left\langle X,Y \right\rangle_A:=&\mathrm{Tr}\left( \left(A^{-1}X\right)^{\top}\I_{p,q}A^{-1}Y \right)\quad\forall X,Y\in T_A O \left( p,q \right)\\
  =&\mathrm{Tr}\left( X^{\top}\I_{p,q}Y \right)\qquad \textrm{since $A^{\top}\I_{p,q}A=I_{p,q}$.}
  \end{aligned}
\end{equation}
For the ease of notation, we shall drop the sub-script $A\in G$ unless there is a potential risk of confusion. This semi-Riemannian metric will be referred to as the \emph{natural} semi-Riemannian metric on $O \left( p,q \right)$. The degenerate bundle of this semi-Riemannian structure can be easily determined as follows. Let $\Delta\in\mathbb{R}^{n\times n}$ be a skew-symmetric matrix such that $A\I_{p,q}\Delta\in T_A O \left( p,q \right)$ for an arbitrary $A\in O \left( p,q \right)$. Setting
\begin{equation*}
  0=\mathrm{Tr}\left( X^{\top}A^{-\top}\I_{p,q}A^{-1}\Delta \right)=\mathrm{Tr}\left( X^{\top}\I_{p,q}\Delta \right)\quad\forall X\in\mathbb{R}^{n\times n}, X^{\top}+X=0
\end{equation*}
we have that $\I_{p,q}\Delta$ must be symmetric, i.e.
\begin{equation}
  \label{eq:symmetry-condition}
  \I_{p,q}\Delta =\Delta^{\top}\I_{p,q}(=-\Delta\I_{p,q})
\end{equation}
Writing $\Delta$ in the partitioned form
\begin{equation*}
  \Delta=
  \begin{bmatrix}
    \Delta_1 & \Delta_2\\
    -\Delta_2^{\top} & \Delta_3
  \end{bmatrix}
\end{equation*}
where
\begin{equation*}
  \Delta_1\in\mathbb{R}^{p\times p}, \quad \Delta_2\in\mathbb{R}^{p\times q}, \quad \Delta_3\in\mathbb{R}^{q\times q}
\end{equation*}
satisfying
\begin{equation*}
  \Delta_1+\Delta_1^{\top}=0,\,\,\Delta_3+\Delta_3^{\top}=0.
\end{equation*}
Plugging this partitioned form into \cref{eq:symmetry-condition} gives
\begin{equation*}
  \Delta_1=0,\quad \Delta_3=0
\end{equation*}
from which it follows that the degenerate bundle of $O \left( p,q \right)$ takes the form
\begin{equation*}
  \begin{aligned}
    O \left( p,q \right)^{\perp}&=\bigcup_{A\in O \left( p,q \right)}\left\{ A\I_{p,q}\begin{bmatrix} & \Delta_2\\ -\Delta_2^{\top} & \end{bmatrix} \,\Bigg|\, \Delta_2\in\mathbb{R}^{p\times q} \right\}\\
    &=\bigcup_{A\in O \left( p,q \right)}\left\{ A\begin{bmatrix} & \Delta_2\\ \Delta_2^{\top} & \end{bmatrix} \,\Bigg|\, \Delta_2\in\mathbb{R}^{p\times q} \right\}
  \end{aligned}
\end{equation*}
In particular, this indicates that the natural semi-Riemannian structure on $O \left( p,q \right)$ is degenerate. By checking at the identity it is clear that $\left[ \mathfrak{o}\left( p,q \right), \mathfrak{o}\left( p,q \right) \right]\nsubseteq \mathfrak{o}\left( p,q \right)$, hence the degenerate bundle $O \left( p,q \right)^{\perp}$ is not integrable. It then follows from \cite[Corollary 3.6]{Kupeli1987c} that $O \left( p,q \right)$ equipped with the natural semi-Riemannian metric does not admit a Koszul derivative.
\end{example}

\begin{example}[Orthogonal Group]
  \label{exm:orth-group}
  Let $\mathbb{N}\ni n=p+q$, $0\leq p\leq n$, and $p,q\in \mathbb{N}$. The manifold structure on the orthogonal group $O \left( n \right)$ is well-known:
  \begin{equation*}
    \begin{aligned}
      O \left( n \right) &= \left\{ A\in\mathbb{R}^{n\times n} \mid AA^{\top}=A^{\top}A=I_{n}\right\}\\
      \mathfrak{o}\left( n \right)&=\left\{ X\in\mathbb{R}^{n\times n}\mid X+X^{\top} =0\right\}\\
      T_AO \left( n \right)&=\left\{ AX\mid X\in\mathbb{R}^{n\times n}, X+X^{\top}=0 \right\}=A\mathfrak{o}\left( n \right),\quad\forall A\in O \left( n \right).
    \end{aligned}
  \end{equation*}
  Equip $O \left( n \right)$ with the same left-invariant semi-Riemannian metric as in \cref{exm:indef-orth-group}:
  \begin{equation*}
    \left\langle U,V \right\rangle_A:=\mathrm{Tr}\left( \left( A^{-1}U \right)^{\top}\I_{p,q}A^{-1}V \right)\quad\forall U,V\in T_AO \left( n \right).
  \end{equation*}
  In this example, again the semi-Riemannian metric is degenerate. In fact, by a similar argument as in \cref{exm:indef-orth-group} one has
  \begin{equation}
    \label{eq:semi-riem-orth-group}
    O \left( n \right)^{\perp}=\bigcup_{A\in O \left( p,q \right)}\left\{ A\begin{bmatrix} & \Delta_2\\ -\Delta_2^{\top} & \end{bmatrix} \,\Bigg|\, \Delta_2\in\mathbb{R}^{p\times q} \right\}
  \end{equation}
  and again, $O \left( n \right)^{\perp}$ is not integrable. In fact, one can also easily verify that the involution $\left[ O \left( n \right)^{\perp}, O \left( n \right)^{\perp} \right]$ is orthogonal to $O \left( n \right)^{\perp}$ with respect to the natural Riemannian (not semi-Riemannian!) metric on $O \left( n \right)$. Again, from \cite[Corollary 3.6]{Kupeli1987c} we know that $O \left( n \right)$ with semi-Riemannian metric \cref{eq:semi-riem-orth-group} does not admit a Koszul derivative.
\end{example}

\begin{remark}
  \cref{exm:orth-group} is a special case of a more general practice: one can equip $O \left( p,q \right)$ with a semi-Riemannian structure with $\I_{p,q}$ replaced with $\I_{p',q'}$ in \cref{exm:indef-orth-group}, where $p+q=p'+q'$ but $p\neq p'$ and $q\neq q'$. Again this is due to the triviality of the tangent bundle of the Lie group $O \left( p,q \right)$ for any integers $p$ and $q$.
\end{remark}

\begin{remark}
  It is natural to ask at this point whether a given manifold of interest, such as $\O \left( p,q \right)$ or $\O \left( n \right)$, admits a semi-Riemannian structure of a particular type for which a Koszul derivative exists. We are not aware of general results of this sort. Some related work (e.g. \cite{Nomizu1979,Barnet1989,Albuquerque1998}) have been devoted to the existence of left-invariant Lorentz metrics satisfying certain curvature sign conditions, following the seminal work of Milnor \cite{Milnor1976}. Bi-invariant semi-Riemannian metrics on Lie groups have also been widely explored since the 1910s; see \cite[\S1.4]{MiolanePennec2015} for a brief survey.
\end{remark}

\begin{remark}\label{rem:semi-riem-decomp}
  Though the notion of orthogonality breaks down for degenerate semi-Riemannian submanifolds, the tangent bundle of any semi-Riemannian manifold of type $\left( \kappa,\nu,\pi \right)$ admits a direct sum decomposition $T\!M=M^{\perp}\oplus H$, where $H$ is a sub-bundle of $T\!M$ with rank $\nu+\pi$. In this case, the restriction of the semi-Riemannian metric on $H$ gives rise to a non-degenerate semi-Riemannian metric of type $\left( 0,\nu,\pi \right)$. We will fully leverage this partial non-degeneracy in the semi-Riemannian optimization algorithm presented in this paper.
\end{remark}

\subsubsection{Gradient and Hessian of Submanifolds of Minkowski Spaces}
\label{sec:grad-hess-subm}

When $M$ is a non-degenerate semi-Riemannian submanifold of a Minkowski space $\mathbb{R}^{p+q}$, gradient and Hessian of a twice differentiable function $f$ on $M$ can be computed explicitly from the gradient and Hessian of $f$ on the ambient Minkowski space $\mathbb{R}^{p+q}$, thanks to the non-degeneracy which ensures for any $x\in M$ that the tangent space $T_xM$ has an orthogonal complement in $\mathbb{R}^{p+q}$, and thus $Df$ and $D^2f$ on $\mathbb{R}^{p+q}$ can be orthogonally projected onto $T_xM$. Specifically, the same argument as in \cite[\textsection 3.6.1]{AMS2009} indicates that the semi-Riemannian gradient of $f$ on $M$ is exactly the orthogonal projection to the tangent space of $M$ of the semi-Riemannian gradient of $f$ as a function defined on the Minkowski space $\mathbb{R}^{p+q}$; a similar argument yields the fact that the Hessian of $f$ is the composition of the Hessian of $f$ on the Minkowski space $\mathbb{R}^{p+q}$ composed with the orthogonal projection from $\mathbb{R}^{p+q}$ to the tangent space of $M$.

\begin{example}[Euclidean Sphere in Minkowski Spaces]
\label{exm:euc-sphere-minkowski}
Consider the standard Euclidean sphere
\begin{equation*}
  \mathbb{S}^{p+q-1}=\left\{ x\in\mathbb{R}^{p,q}\mid x_1^2+\cdots+x_{p+q}^2=1 \right\}
\end{equation*}
as a submanifold of $\mathbb{R}^{p,q}$, the Minkowski space equipped with inner product $\I_{p,q}\in\mathbb{R}^{p+q}$ as defined in \cref{exm:euc-pq}. For any $x\in\mathbb{S}^{p+q-1}$, the tangent space $T_x\mathbb{S}^{p+q-1}$ can be specified as
\begin{equation*}
  T_x\mathbb{S}^{p+q-1}=\left\{ v\in \mathbb{R}^{p,q}\mid v^{\top}x=\left\langle v,\I_{p,q}x \right\rangle =0\right\}
\end{equation*}
and thus the projection from $\mathbb{R}^{p,q}$ to $T_x\mathbb{S}^{p+q-1}$ is
\begin{equation*}
  P_x \left( v \right):=v-\frac{\left\langle v,\I_{p,q}x \right\rangle}{\left\langle \I_{p,q}x,\I_{p,q}x \right\rangle}\I_{p,q}x=v-\frac{v^{\top}x}{x^{\top}\I_{p,q}x}\I_{p,q}x,\quad\forall x^{\top}\I_{p,q}x\neq 0.
\end{equation*}
For $x\in\mathbb{S}^{p+q-1}$ with $x^{\top}\I_{p,q}x=0$, the projection operator $P_x$ is not defined since $x$ is a null vector. Nevertheless, this occurs only for a set of measure zero on $\mathbb{S}^{p+q-1}$, which means they almost never occur in practice. In our numerical experiments on $\mathbb{S}^{p+q-1}$ (see \Cref{sec:euclidean-spheres}), we just randomly perturb the point $x$ so that the optimization trajectory stays away from the degenerate locus. This works perfectly as long as the optimum is not on the degenerate locus. If unavoidable, we can also temporarily resort to the Riemannian orthogonal projection for $x\in\mathbb{S}^{p+q-1}$ with $x^{\top}\I_{p,q}x=0$. For a twice differentiable function $f:\mathbb{S}^{p+q-1}\rightarrow\mathbb{R}$, if we denote $Df$ and $D^2f$ for the semi-Riemannian gradient and Hessian of $f$ on the ambient Minkowski space (following \cref{exm:semi-riem-euc}), then the semi-Riemannian gradient and Hessian of $f$ on $\mathbb{S}^{p+q-1}$ are $P_x \left( Df \left( x \right) \right)$ and $P_x \left( D^2 f \left( x \right) \right)$, respectively.
\end{example}

\subsubsection{Geodesics and Parallel-Transports}
\label{sec:geod-parall-transp}

Regardless of whether the semi-Riemannian submanifold under consideration is degenerate, we can define analogies of geodesics and parallel-transports on them by means of their \emph{semi-normal bundles}. To this end, for semi-Riemannian manifold $M$ and its submanifold $X$ we denote by $TM$ and $TX$ the tangent bundles of $M$ and $X$, respectively. Let $x\in X$, we define the \emph{semi-normal space} of $X$ in $M$ at $x$ to be 
\[
\SN_x(X, M) \coloneqq \{u \in T_x M: \widetilde{g}_x(u,v)=0,~\forall v\in T_x X \}.
\] 
We also define the \emph{semi-normal distribution} of $X$ in $M$ to be 
\[
\SN(X,M) \coloneqq \bigsqcup_{x\in X} \SN_x(X,M).
\]
Consider the linear map 
\[
\SN_x(X,M) \hookrightarrow \T_x M \to \T_xM/ \T_xX
\]
where the first map is the inclusion and the second map is the quotient map. The following observation is straightforward by definition.

\begin{lemma}\label{lemma:SN v.s. X perp}
Fibres of the degenerate bundle of $X$ (c.f. \cref{def:degenerate bundle}) at $x\in X$ can be written as
\[
X_x^\perp = \SN_x(X,M) \cap \T_x X.
\]
In particular, if $X$ is an open submanifold of $M$, then $X$ is a non-degenerate semi-Riemannian submanifold of $M$. 
\end{lemma}
Hence we have an injective map
\[
\SN_x(X,M)/X_x^\perp \hookrightarrow \T_xM/ \T_xX
\]
and thus $\SN(X,M)/X^\perp$ is a sub-distribution of the \emph{normal bundle} $\N(X,M)\coloneqq  \T M|_X / \T X$. If $\dim \SN_x(X,M) - \dim X_x^\perp$ is constant with respect to $x\in X$, then $\SN(X,M)$ is a sub-bundle of $\N(X,M)$ and will be referred to as the \emph{semi-normal bundle} of $X$ with respect to $M$. We define the analogy of geodesics on (possibly degenerate) semi-Riemannian submanifolds as curves with accelerations in the semi-normal bundle --- when the semi-Riemannian submanifold becomes non-degenerate these geodesics reduces to standard semi-Riemannian geodesics.
% \begin{lemma}
% If $\dim \SN_x(X,M) - \dim X_x^\perp$ is a constant for every $x\in X$, then $\SN(X,M)$ is a sub-bundle of $\N(X,M)$. 
% \end{lemma}

\begin{definition}\label{defn:geodesic}
For a given $x\in X$ and $\Delta\in \T_x X$, if a smooth curve $\gamma:[-\epsilon,\epsilon] \to X$ satisfies $\gamma(0) = x, \dot{\gamma}(0) = \Delta$ and 
\[ 
\frac{D}{dt}(\dot{\gamma}(t)) \in \SN_{\gamma(t)}(X,M)
\]
for all $t\in [-\epsilon,\epsilon]$, then $\gamma$ is called an embedded geodesic curve on $X$ passing through $x$ with the tangent direction $\Delta$. Here $\frac{D}{dt}(\dot{\gamma}(t))$ is the covariant derivative of $\dot{\gamma}(t)$ along $\gamma(t)$ on the ambient semi-Riemannian manifold $(M,g)$. 
\end{definition}

\begin{definition}\label{defn:parallel transport}
Let $\gamma(t)$ be a curve passing through $x=\gamma(0)$ on $X$ and let $\Delta\in \T_x X$ be a given tangent vector. A parallel transportation of $\Delta$ along the curve $\gamma(t)$ is a vector field $\Delta(t)$ such that $\Delta(0) = \Delta$ and 
\begin{equation*}
\frac{D}{dt}\left(\Delta(t)\right) \in \SN_{\gamma(t)} (X,M).
\end{equation*}
\end{definition}
We remark that on a (semi-)Riemannian manifold $(Z,g)$, a geodesic $\tau$ passing through $z\in Z$ with the tangent direction $U\in \T_z Z$ is traditionally defined by the second order ODE with initial condition:
\begin{equation}\label{eqn:geodesic via covariant derivative}
\begin{cases}
\nabla_{\dot{\tau}(t)} \dot{\tau}(t) = 0, \\
\tau(0) = x,\,\,\dot{\tau}(0) = U
\end{cases}
\end{equation}
where $\nabla$ is the covariant derivative uniquely determined by the metric $g$. In the meanwhile, a well-known fact (cf. \cite[Corollary 10]{ONeill1983}) is that if $(Z,g)$ is isometrically embedded in a (semi-)Riemannian manifold $(\overline{Z},\overline{g})$ then \cref{eqn:geodesic via covariant derivative} is equivalent to the condition that $\overline{D}/dt (\dot{\gamma}(t))$ is always perpendicular to $Z$, i.e. 
\begin{equation*}
\overline{g}\left(\frac{\overline{D}}{dt} \left(\dot{\gamma}(t)\right), V\right) =0
\end{equation*}
for all $V\in \T_{\gamma(t)} Z$. Here $\overline{D}$ is the covariant derivative on $\overline{Z}$ along the curve $\gamma(t)$. From this second perspective, \cref{defn:geodesic} and \cref{defn:parallel transport} are natural generalizations of geodesics and parallel-transports from nondegenerate to degenerate semi-Riemannian geometry. Of course, it is in general not possible to obtain closed-form expressions for the embedded geodesic curves and parallel-transports; see \Cref{sec:semi-riem-geom-matrix-lie-groups} for some examples. These definitions apply to the particular case when the semi-Riemannian structure under consideration is actually Riemannian, and thus the optimization methods are also applicable to degenerate Riemannian manifolds.

\subsection{Semi-Riemannian Hypersurfaces of Minkowski Spaces}
\label{sec:semi-riemannian-sub}

In this subsection we describe the semi-Riemannian geometry of submanifolds of codimension one in the Minkowski space $\mathbb{R}^{p,q}$ (see \cref{exm:euc-pq}), which are prototypical examples of semi-Riemannian manifolds. Throughout this subsection $X$ denotes a submanifold of $\mathbb{R}^{p,q}$. Unraveling the definition of semi-normal and normal bundles yields:
\begin{proposition}\label{prop:normal v.s. semi-normal 1}
For each $x\in X$, we have 
\[
\SN_x(X,\mathbb{R}^{p,q}) = \I_{p,q} \N_x(X,\mathbb{R}^{p+q})
\]
where
\begin{equation*}
  \N_x(X,\mathbb{R}^{p+q})=\left\{ v\in T_x\mathbb{R}^{p+q}\mid v_1w_1+\cdots+v_{p+q}w_{p+q}=0\textrm{ for all }w\in T_xX \right\}
\end{equation*}
and
\begin{equation*}
  \I_{p,q} \N_x(X,\mathbb{R}^{p+q})=\left\{ \I_{p,q}v\mid v\in N_x(X,\mathbb{R}^{p+q}) \right\}.
\end{equation*}
In particular, $\SN(X,\mathbb{R}^{p,q})$ is a vector bundle on $X$ of rank $(p+q - \dim X)$.
\end{proposition}

\begin{corollary}
Let $X\subseteq \mathbb{R}^{p,q}$ be a \emph{hypersurface} (submanifolds of co-dimension one), and $x\in X$. Then either $X_x^\perp = \{0\}$ or $X_x^\perp = \SN_x(X,\mathbb{R}^{p,q}) = \I_{p,q} \N_x(X,\mathbb{R}^{p+q})$.
\end{corollary}
\begin{proof}
  By \cref{lemma:SN v.s. X perp} we have $X_x^\perp=\SN_x(X,\mathbb{R}^{p,q})\cap T_xX$, but by \cref{prop:normal v.s. semi-normal 1}, we know $\SN_x(X,\mathbb{R}^{p,q})=\I_{p,q} \N_x(X,\mathbb{R}^{p+q})$ is one-dimensional.
\end{proof}

\begin{example}[Euclidean Spheres in Minkowski Spaces]
\label{exm:euc-sphere}
Consider as in \cref{exm:euc-sphere-minkowski} the hypersurface
\begin{equation*}
  \mathbb{S}^{p+q-1}=\left\{ x\in\mathbb{R}^{p,q}\mid x_1^2+\cdots+x_{p+q}^2=1 \right\}\subseteq \mathbb{R}^{p,q}.
\end{equation*}
Direct calculation yields
\begin{equation*}
  \begin{aligned}
    \N_x(\mathbb{S}^{p+q-1},\mathbb{R}^{p+q}) &= \left\lbrace \lambda x: \lambda\in \mathbb{R} \right\rbrace,\\
    \SN_x(\mathbb{S}^{p+q-1},\mathbb{R}^{p,q}) &= \I_{p,q}\N_x(\mathbb{S}^{p+q-1},\mathbb{R}^{p+q})=\left\lbrace \lambda \I_{p,q} x: \lambda\in \mathbb{R}\right\rbrace\\
    T_x \mathbb{S}^{p+q-1}&=\left\{ v\in\mathbb{R}^{p,q}\mid v_1x_1+\cdots+v_{p+q}x_{p+q}=0 \right\}
  \end{aligned}
\end{equation*}
and thus
\begin{equation*}
  \begin{aligned}
    (\mathbb{S}^{p+q-1})_x^\perp &= \SN_x(\mathbb{S}^{p+q-1},\mathbb{R}^{p,q}) \cap T_x \mathbb{S}^{p+q-1}\\
&=\left\{ \lambda\I_{p,q}x\mid \lambda\in\mathbb{R},\,\,x_1^2+\cdots+x_p^2=x_{p+1}^2+\dots+x_{p+q}^2 \right\}\\
&=
\begin{cases}
\displaystyle \SN_x(\mathbb{S}^{p+q-1},\mathbb{R}^{p,q})=\I_{p,q}\N_x(\mathbb{S}^{p+q-1},\mathbb{R}^{p+q}),&\text{if}~x^{\top}\I_{p,q}x=0,\\
\left\lbrace 0 \right\rbrace,&\text{otherwise}.
\end{cases}
  \end{aligned}
\end{equation*}
%where $\mathbb{S}^{p,q}:=\left\{ x\in\mathbb{R}^{p,q}\mid -x_1^2-\cdots-x_p^2+x_{p+1}^2+\cdots+x_{p+q}^2=1 \right\}$ is the unit \emph{pseudo-sphere} in the Minkowski space $\mathbb{R}^{p,q}$; see \Cref{sec:pseudo-spheres} below.
\end{example}

It is conceivable that hypersurfaces, and in particular those linear ones --- known as \emph{hyperplanes} --- play an important role in semi-Riemannian geometry as they can provide rich yet elementary examples of non-degenerate semi-Riemannian sub-manifolds. In fact, generically speaking, hyperplanes inherit non-degenerate semi-Riemannian structures from the ambient Minkowski spaces; we defer a simple proof to supplementary materials. It makes use of a handy criterion for the non-degeneracy of semi-Riemannian structures on hypersurfaces which we establish as follows. First of all, we point out that \cref{prop:normal v.s. semi-normal 1} can be equivalently interpreted in terms \emph{Gauss maps}: for closed sub-manifolds $X\subset \mathbb{R}^{p,q}$ with $\operatorname{dim} X<p+q$, denote $m:= p + q - \operatorname{dim}\left( X \right)$ and define the \emph{Gauss map} 
\[
\N:X\to \Gr(m,p+q),\quad \N(x) = \N_x(X,\mathbb{R}^{p+q}).
\]
and the \emph{semi-Gauss map}
\[
\SN: X \to \Gr(m,p+q),\quad \SN(x)= \SN_x(X,\mathbb{R}^{p,q}).
\]
\cref{prop:normal v.s. semi-normal 1} states essentially the commutativity of the following diagram:
\begin{equation}\label{diag:normal v.s. semi-normal}
     \begin{tikzcd}
   X \arrow{r}{\SN} \arrow{d}[anchor=center,xshift=-2ex,yshift = 0ex]{\N}  &  \Gr(m,p+q)\\
   \Gr(m,p+q) \arrow{ru}[anchor=center,xshift=3ex,yshift = -1.5ex]{\I_{p,q}} 
     \end{tikzcd}
\end{equation}
Denote by $\mathcal{V}$ the quadratic hypersurface in $\mathbb{P}\mathbb{R}^{p+q}$ defined by
\begin{equation*}
  \mathcal{V}:=\left\{ \left( x_1,\cdots,x_p,y_1,\cdots,y_q \right)\in \mathbb{P}\mathbb{R}^{p+q}\,\Bigg|\,\sum_{j=1}^p x_j^2 - \sum_{j=1}^q y_j^2 = 0 \right\}.
\end{equation*}
The degeneracy of semi-Riemannian structures on hypersurfaces is totally determined by the intersections of semi-normal bundles with $\mathcal{V}$. More concretely, it follows directly from the definitions that
\begin{proposition}\label{prop:criterion-non-degeneracy}
If $\dim X = p+q -1$, then $\N^{-1}(\mathcal{V})$ is the degenerate locus of $X$. In particular, $X$ is non-degenerate if and only if $\mathcal{V} \cap \N(X) = \emptyset$, where $\N(X)$ is the image of the Gauss map of $X$. 
\end{proposition}

In the remainder of this section we provide two classes of hypersurfaces, namely, \emph{pseudo-spheres} and \emph{pseudo-hyperbolic spaces}, in the Minkowski space $\mathbb{R}^{p,q}$ that are different from hyperplanes. For both examples we obtain closed form expressions for embedded geodesic curves and parallel-transports (see \cref{defn:geodesic} and \cref{defn:parallel transport}) needed for implementing the algorithmic framework proposed in \Cref{sec:semi-riem-optim}. Numerical experiments demonstrating the efficacy of the semi-Riemannian optimization framework on these hypersurfaces can be found in \Cref{sec:numer-exper}.

\subsubsection{Pseudo-spheres}
\label{sec:pseudo-spheres}
Let $\mathbb{S}^{p,q}$ be the hypersurface in $\mathbb{R}^{p,q}$ defined by the equation 
\[
-\sum_{j=1}^p x_j^2 + \sum_{j=1}^q y_j^2 = 1.
\]
Here we write $x\in \mathbb{R}^{p,q}$ as $x = (x_1,\dots, x_p,y_1,\dots, y_q)$. In the literature, $\mathbb{S}^{p,q}$ is called the \emph{unit pseudo-sphere} in $\mathbb{R}^{p,q}$, and $\mathbb{S}^{1,q}$ (resp. $\mathbb{S}^{p,1}$) is known asx the de Sitter (resp. Anti-de Sitter) space. The tangent space $\T_x \mathbb{S}^{p,q}$ is characterized by 
\[
\T_x \mathbb{S}^{p,q} = \left\lbrace
(u,v)\in \mathbb{R}^{p,q}: -\sum_{j=1}^p x_j u_j + \sum_{j=1}^q y_j v_j = 0
\right\rbrace
\]
for each $x = (x_1,\dots, x_p,y_1,\dots,y_q)\in \mathbb{R}^{p,q}$. Hence we also have 
\[
\N_x(\mathbb{S}^{p,q},\mathbb{R}^{p,q}) =
\left\lbrace
\lambda \I_{p,q} x :\lambda\in \mathbb{R}
\right\rbrace,\quad \SN_x(\mathbb{S}^{p,q},\mathbb{R}^{p+q}) =\left\lbrace
\lambda x:\lambda\in \mathbb{R}
\right\rbrace.
\]
This together with \cref{prop:criterion-non-degeneracy} implies the following
\begin{lemma}\label{lemma:nondegeneracy pseudo-sphere}
For any positive integers $p$ and $q$, $\mathbb{S}^{p,q}$ is a non-degenerate semi-Riemannian sub-manifold of $\mathbb{R}^{p,q}$.
\end{lemma}

We now turn to investigating the embedded geodesics on $\mathbb{S}^{p,q}$.
\begin{proposition}\label{prop:geodesics psudo-sphere}
The embedded geodesic passing through $x\in \mathbb{S}^{p,q}$ with tangent direction $X\in T_x\mathbb{S}^{p,q}$ is 
\begin{equation}
\label{eq:semi-riem-geod-pseudo-sphere}
\gamma(t) = \begin{cases}
\displaystyle x \cos (t\left\| X \right\|) + \left( X/\left\| X \right\| \right) \sin(t \left\| X \right\|),&\text{if}~\langle X,X \rangle > 0,\\
\displaystyle x\cosh(t\left\| X \right\|) + \left( X/\left\| X \right\| \right) \sinh(t\left\| X \right\|),&\text{if}~\langle X,X \rangle < 0,\\
\displaystyle x + t X,&\text{otherwise}
\end{cases}
\end{equation}
where $\left\| X \right\| = \sqrt{\left| \langle X, X \rangle_x \right| }$.
\end{proposition}
\begin{proof}
First, it is straightforward to verify that $\gamma(0) = x$ and $\dot{\gamma}(0) = X$. Next we notice that 
\[
\gamma(t)^\mathsf{T} \I_{p,q} \gamma(t) = 1 
\]
since $x^\mathsf{T}\I_{p,q} X = 0$. This implies that $\gamma(t)$ is indeed a curve on $\mathbb{S}^{p,q}$. Lastly, by taking second derivative, we have 
\[
\ddot{\gamma}(t) = \begin{cases}
-\langle X,X \rangle \gamma(t),~\text{if}~\langle X,X \rangle \ne 0, \\
0,~\text{otherwise}\\
\end{cases}
\]
and hence $\ddot{\gamma}(t)\in \SN_{\gamma(t)}(\mathbb{S}^{p,q},\mathbb{R}^{p,q})$. Therefore, $\gamma(t)$ is the geodesic curve passing through $x$ with tangent direction $X$.
\end{proof}

We now compute the parallel translation on $\mathbb{S}^{p,q}$. Let $x$ be a point on $\mathbb{S}^{p,q}$ and let $\Delta$ be a tangent vector on $\mathbb{S}^{p,q}$ at $x$. We denote by $\gamma(t)$ the geodesic curve passing through $x$ with tangent direction $X$. Let $ \Delta(t)$ be the parallel transportation of $\Delta$ along $\gamma$. By definition, we must have that $\Delta(t)\in \T_{\gamma(t)} \mathbb{S}^{p,q}$ and $\dot{\Delta}(t)\in \SN_{\gamma(t)}(\mathbb{S}^{p,q},\mathbb{R}^{p,q})$. This implies 
\begin{align}
   \label{eqn:parallel trans Spq 1}
   \langle \Delta(t),\gamma(t) \rangle &= 0,\\ 
    \label{eqn:parallel trans Spq 2}
\langle \dot{\Delta}(t),\gamma(t) \rangle \gamma(t)&= \dot{\Delta}(t). 
\end{align}
Differentiating \cref{eqn:parallel trans Spq 1}, we obtain $\langle \dot{\Delta}(t), \gamma \rangle = -\langle \Delta(t),\dot{\gamma}(t) \rangle$ and hence 
\[
\dot{\Delta}(t) = - \langle \Delta(t),\dot{\gamma}(t) \rangle \gamma(t).
\]
Since parallel translation preserves inner product, we see that $\langle \Delta(t),\dot{\gamma}(t) \rangle = \langle \Delta, X \rangle$ and 
\begin{equation}\label{eqn:parallel trans Spq 3}
   \dot{\Delta}(t) = - \langle \Delta,X \rangle \gamma(t). 
\end{equation}
Integrating \cref{eqn:parallel trans Spq 3} and using the initial condition that $\Delta(0) = \Delta$ to get
\begin{proposition}\label{prop:parallel transport psudo-sphere}
Let $\gamma(t)$ be the geodesic passing through $x\in \mathbb{S}^{p,q}$ with tangent direction $X\in \T_x \mathbb{S}^{p,q}$. The parallel transport of $\Delta\in \T_x\mathbb{S}^{p,q}$ along the $\gamma(t)$ is 
\[
\Delta(t) = -\langle \Delta, X \rangle \int_0^t \gamma(\tau) d\tau + \Delta.
\]
More precisely, we have 
\begin{equation*}
  \begin{aligned}
    &\Delta(t) = \\
&\begin{cases}
\displaystyle -\frac{\langle \Delta, X \rangle}{\left\| X \right\|} \left[x\sin\left(t\left\| X \right\|\right) - \frac{X}{\left\| X \right\|}\cos(t\left\| X \right\|)\right] + \left(\Delta - \frac{\langle \Delta,X\rangle}{\left\| X \right\|^2}X\right),&\!\!\!\!\!\!\text{if}~\langle X,X \rangle > 0, \\
&\\
\displaystyle-\frac{\langle \Delta, X \rangle}{\left\| X \right\|} \left[x \sinh\left(t\left\| X \right\|\right) + \frac{X}{\left\| X \right\|} \cosh(t\left\| X \right\|)\right] + \left(\Delta + \frac{\langle \Delta,X \rangle}{\left\| X \right\|^2}X\right),&\!\!\!\!\!\!\text{if}~\langle X, X\rangle <0, \\
&\\
\displaystyle-\langle \Delta, X \rangle \left(tx + \frac{1}{2} t^2 X\right) + \Delta,&\!\!\!\!\!\!\text{otherwise}.
\end{cases}
  \end{aligned}
\end{equation*}
\end{proposition}

\subsubsection{Pseudo-hyperbolic Spaces}
\label{sec:pseudo-hyperb-spac}

The unit pseudo-hyperbolic space $\mathbb{H}^{p,q}$ in $\mathbb{R}^{p,q}$ is defined by the equation 
\[
-\sum_{j=1}^p x_j^2 + \sum_{j=1}^q y_j^2 = -1.
\]
The tangent space of $\mathbb{H}^{p,q}$ at a point $x = (x_1,\dots, x_p,y_1,\dots, y_q)$ is 
\[
\T_x \mathbb{H}^{p,q} = \lbrace (u,v)\in \mathbb{R}^{p,q}: -\sum_{j=1}^p x_j u_j + \sum_{j=1}^q y_j v_j = 0 \rbrace
\]
Let $\sigma_{p,q}:\mathbb{R}^{p,q} \to \mathbb{R}^{q,p}$ be the map defined by 
\[
\sigma_{p,q}(x_1,\dots, x_p,y_1,\dots,y_q) = (y_1,\dots,y_q,x_1,\dots, x_p).
\]
It is straightforward (\cite[Lemma 24]{ONeill1983})  to verify that $\sigma_{p,q}$ is an anti-isometry between $\mathbb{H}^{p,q}$ and $\mathbb{S}^{q,p}$, whose inverse is $\sigma_{q,p}$. Therefore we have
\[
\N_x (\mathbb{H}^{p,q},\mathbb{R}^{p,q}) = \lbrace \lambda \I_{p,q}x: \lambda\in \mathbb{R}\},\quad \SN_x(\mathbb{H}^{p,q},\mathbb{R}^{p,q}) = \{\lambda x: \lambda \in \mathbb{R}\rbrace.
\] 
\begin{corollary}\label{cor:nondegeneracy pseudo-hyperbolic}
For any positive integers $p,q$, the pseudo-hyperbolic space $\mathbb{H}^{p,q}$ is a non-degenerate semi-Riemannian sub-manifold of $\mathbb{R}^{p,q}$.
\end{corollary}
Moreover, geodesics and parallel transports on $\mathbb{H}^{p,q}$ can be easily obtained from those on $\mathbb{S}^{q,p}$ via the anti-isometry $\sigma_{p,q}$.
\begin{corollary}\label{cor:geodesic pseudo-hyperbolic}
Let $x$ be a point on $\mathbb{H}^{p,q}$ and let $X$ be a tangent direction of $\mathbb{H}^{p,q}$ at $x$. The geodesic curve $\gamma(t)$ passing through $x$ with tangent direction $X$ is 
\[
\gamma(t) = \begin{cases}
 x\cosh(t\left\| X \right\|) + \left(X/\left\| X \right\|\right)\sinh(t\left\| X \right\|),&\text{if}~\langle X,X \rangle >0,\\
 x\cos(t\left\| X \right\|)) + \left(X/\left\| X \right\|\right)\sin(t \left\| X \right\|),&\text{if}~\langle X, X \rangle <0,\\
x + tX,&\text{otherwise}.
\end{cases}
\]
\end{corollary}
\begin{corollary}\label{cor:parallel transport pseudo-hyperbolic}
Let $\gamma(t)$ be the geodesic on $\mathbb{H}^{p,q}$ passing through $x\in \mathbb{H}^{p,q}$ with tangent direction $X\in \T_x \mathbb{H}^{p,q}$. The parallel transport of $\Delta\in \T_x\mathbb{H}^{p,q}$ along $\gamma(t)$ is 
\begin{equation*}
  \begin{aligned}
    &\Delta(t) = \\
&\begin{cases}
 \displaystyle\frac{\langle\Delta, X \rangle}{\left\| X \right\|} \left[x\sinh(t\left\| X \right\|) + \frac{X}{\left\| X \right\|}\cosh(t\left\| X \right\|)\right] + \left(\Delta - \frac{\langle \Delta, X\rangle}{\left\| X \right\|^2} X\right),&\text{if}~\langle X, X \rangle >0,\\
&\\
\displaystyle \frac{\langle \Delta, X \rangle}{\left\| X \right\|} \left[x\sin\left(t\left\| X \right\|\right) - \frac{X}{\left\| X \right\|}\cos(t\left\| X \right\|) \right] + \left(\Delta + \frac{\langle \Delta,X \rangle}{\left\| X \right\|^2}X\right),&\text{if}~\langle X,X\rangle <0,\\
&\\
\displaystyle\langle \Delta,X\rangle \left(tx + \frac{1}{2}t^2X\right) + \Delta,&\text{otherwise}.
\end{cases}
\end{aligned}
\end{equation*}
\end{corollary}

\section{Numerical Experiments}
\label{sec:numer-exper}

We demonstrate in this section the feasibility of the proposed semi-Riemannian optimization framework through various conceptual or numerical experiments.

\subsection{Minkowski Spaces}
\label{sec:minkowski-spaces}

Although we know from \cref{exm:semi-riem-grad-descent-optim} that the semi-Riemannian descent direction coincides with the negative Riemannian gradient when the standard orthonormal basis is chosen and fixed at every point of $\mathbb{R}^{1,1}$, the two types of gradients nevertheless differ from each other if we follow the random orthonormal basis construction \cref{alg:semi-riem-basis-pursuit}. To illustrate the difference between Riemannian and semi-Riemannian optimization on Minkowski spaces, we solve a simple quadratic convex optimization problem
\begin{equation}
  \label{eq:toy-quadratic}
  \min_{x\in\mathbb{R}^2}x^\top Ax
\end{equation}
on $\mathbb{R}^{1,1}$ equipped with the standard semi-Riemannian metric of signature $\left(-,+\right)$. Here $A\in\mathbb{R}^{2\times 2}$ is a randomly generated symmetric positive definite matrix, and we apply both steepest descent \cref{alg:gd-semi-mfld} and conjugate gradient \cref{alg:cg-semi-mfld}, using random orthonormal bases in subrountine \cref{alg:finding-steepest-descent-dir} for finding descent directions and Armijo's rule for line search. The semi-Riemannian optimization trajectories vary from instances to instances due to the randomness in basis construction, but global convergence to the global minimum $x=\left(0,0\right)^{\top}$ is empirically observed. We illustrate in \cref{fig:minkowski-sd-cg} the comparison among trajectories of Riemannian/semi-Riemannian steepest descent and conjugate gradient algorithms for one random instance.

\begin{figure}[htbp]
  \centering
  \includegraphics[width=0.49\textwidth]{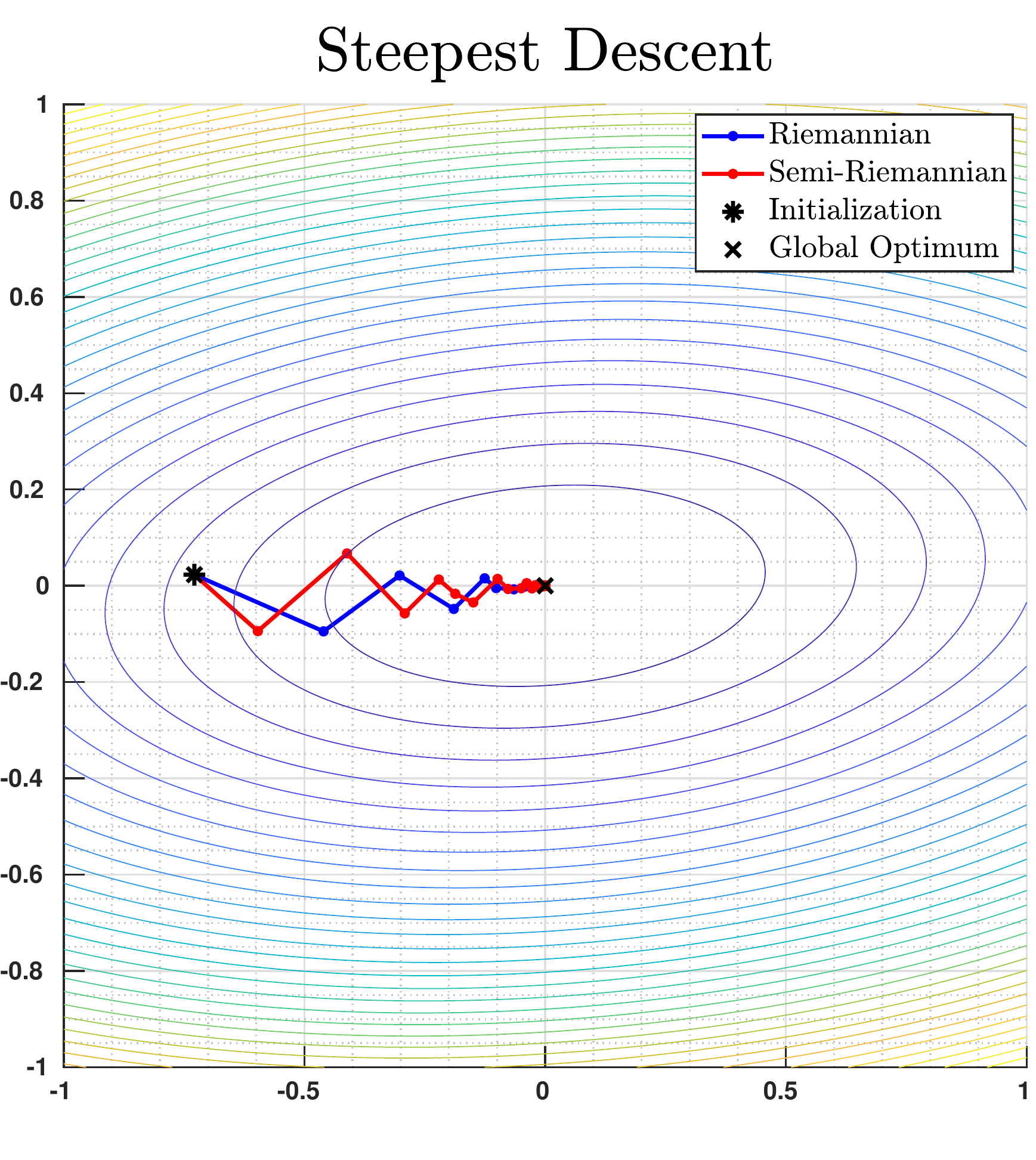}
  \includegraphics[width=0.49\textwidth]{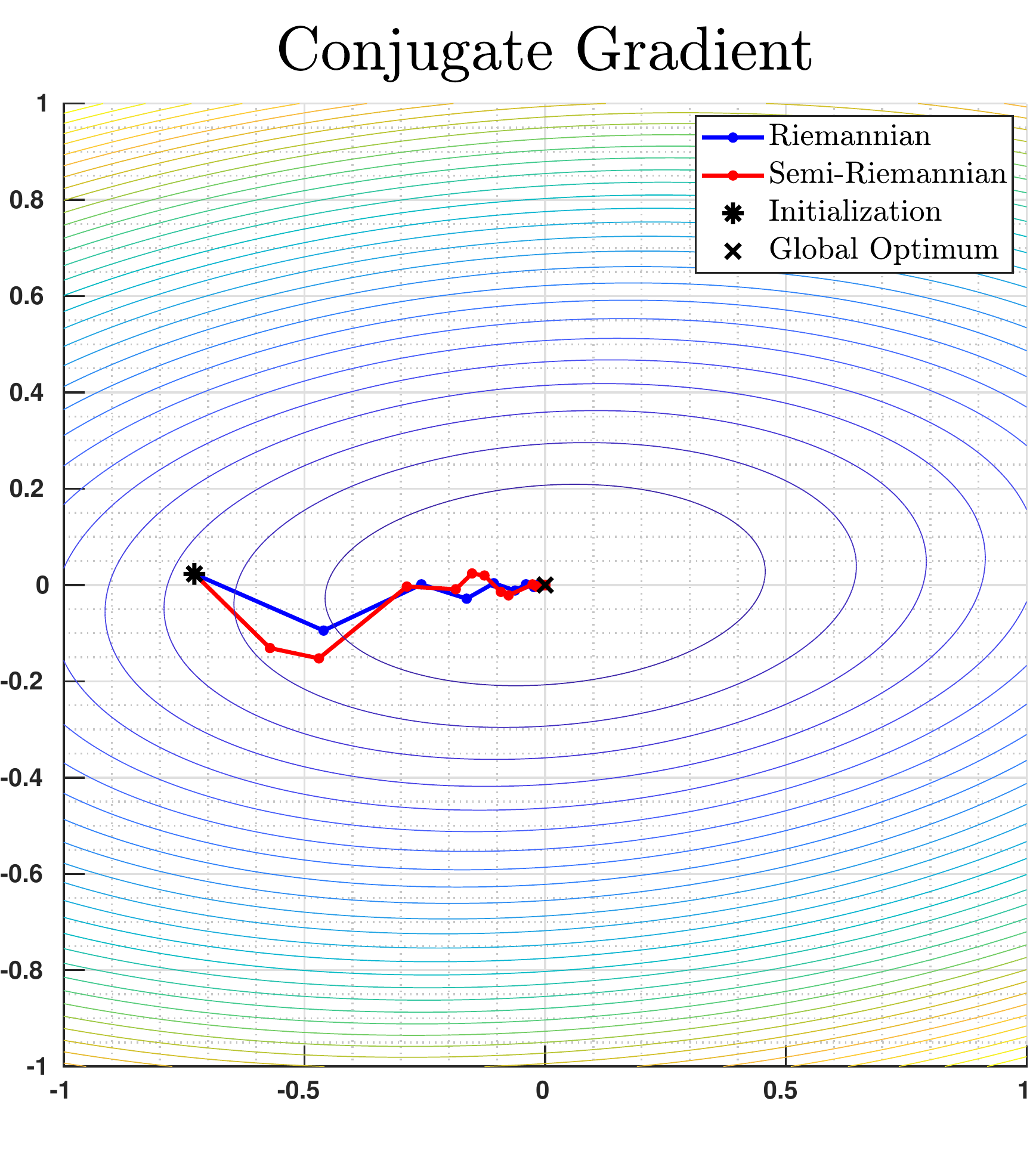}
  \caption{Riemannian and semi-Riemannian steepest descent \cref{alg:gd-semi-mfld} and conjugate gradient \cref{alg:cg-semi-mfld} optimization on the Minkowski Space $\mathbb{R}^{1,1}$ for an instance of the quadratic convex problem \cref{eq:toy-quadratic}, where $A=[0.3649,-0.1065;-0.1065,1.7427]$ and the initial point is chosen as $x_0=\left( -0.7285, 0.0230 \right)^{\top}$.}
  \label{fig:minkowski-sd-cg}
\end{figure}

\subsection{Euclidean Spheres in Minkowski Spaces}
\label{sec:euclidean-spheres}

The calculations in \cref{exm:euc-sphere} imply that the unit Euclidean sphere $\mathbb{S}^{p+q-1}$ is nondegenerate as a semi-Riemannian submanifold in $\mathbb{R}^{p,q}$ except for a degenerate locus of measure zero. Let $f:\mathbb{S}^{p+q-1}\rightarrow\mathbb{R}$ be a differentiable function on $\mathbb{S}^{p+q-1}$, and denote $\nabla f$ for the gradient of $f$ in the ambient Euclidean space. As shown in \cref{exm:semi-riem-euc}, the semi-Riemannian gradient of $f$ in the Minkowski space can be written as $Df=I_{p,q}\nabla f$, and the descent directions in the ambient space take the form $-\left[ Df \right]^{+}=-\nabla f$. Recall from \cref{exm:euc-sphere-minkowski} and \cref{exm:euc-sphere} that, unless $x$ is a null vector (which is a set of measure zero), the fibre of the degenerate bundle $\left( \mathbb{S}^{p+q-1} \right)^{\perp}$ vanishes at $x$ and thus we can project $Df$ to a unique tangent vector in $T_x\mathbb{S}^{p+q-1}$ by \cref{lem:nondegeneracy}. This indicates that the optimization trajectory falls outside of the degenerate locus with probability $1$, as long as the optimum in not inside the degenerate locus. 

To illustrate the feasibility of our proposed semi-Riemannian optimization framework, we solve the problem
\begin{equation}
  \label{eq:euc-sphere-minkowski}
  \max_{x_1^2+\cdots+x_{p+q}^2=1}x^{\top}Ax
\end{equation}
using semi-Riemannian steepest descent and conjugate gradient methods, where $A\in\mathbb{R}^{\left( p+q \right)\times \left( p+q \right)}$ is a randomly generated symmetric (but not necessarily positive definite) matrix. Obviously, the maximum of \cref{eq:euc-sphere-minkowski} is attaned at the eigenvector associated with the maximum eigenvalue of the matrix $A$, and hence we can visualize and compare the convergence dynamics of Riemannian and semi-Riemannian optimization schemes by keeping track of the $L^2$-discrepancy between solutions obtained at each iteration and the true maximizer. As there does not seem to exist explicit expressions for the semi-Riemannian geodesic and parallel-transport on $\mathbb{S}^{p+q-1}$ (see \cref{table:summary}), we use Riemannian geodesic and parallel transport on $\mathbb{S}^{p+q-1}$; generically, these choices do not essentially affect the convergence of manifold optimization algorithms, which allows for arbitrary retractions \cite{AMS2009,BAC2016} and general parallel-transports \cite{RW2012,HAG2015}. Apart from the random basis generation inherent to the local semi-Riemannain Gram-Schmidt orthonormalization \cref{alg:semi-riem-basis-pursuit}, for $p+q>2$ there also exist multiple semi-Riemannian structures on $\mathbb{R}^{p+q}$ which induce distinct semi-Riemannian structures on $\mathbb{S}^{p+q-1}$; our experimental results in \cref{fig:sphere-sd-cg} suggest that all semi-Riemannian structures ensure convergence, though the convergence rates may differ. A deeper investigation of the depenence of convergence rate on the choice of semi-Riemannian structures appears highly intriguing but is beyond the scope of this paper; we defer such exploration to future work.

\begin{figure}[htbp]
  \centering
  \includegraphics[width=0.49\textwidth]{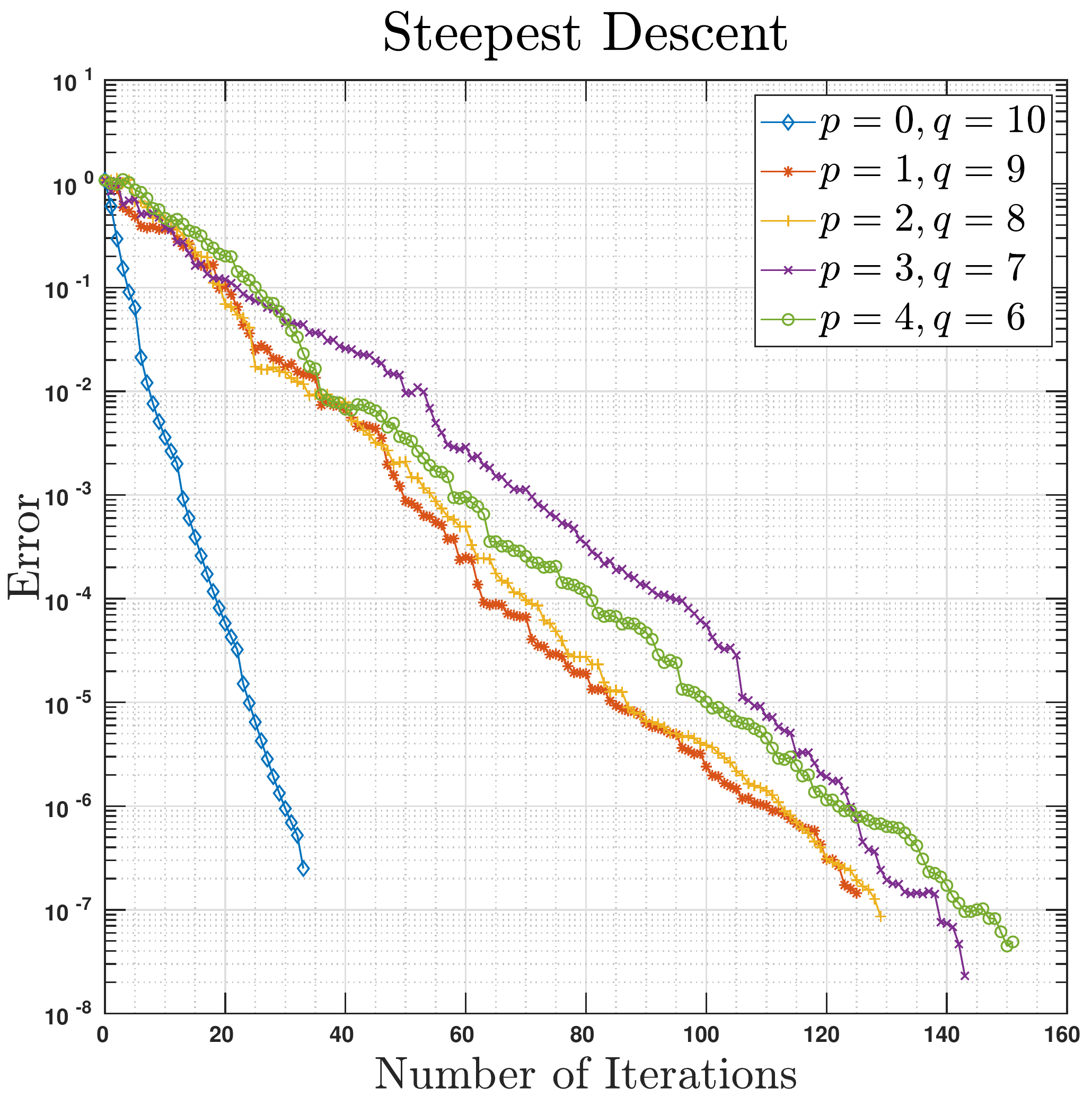}
  \includegraphics[width=0.49\textwidth]{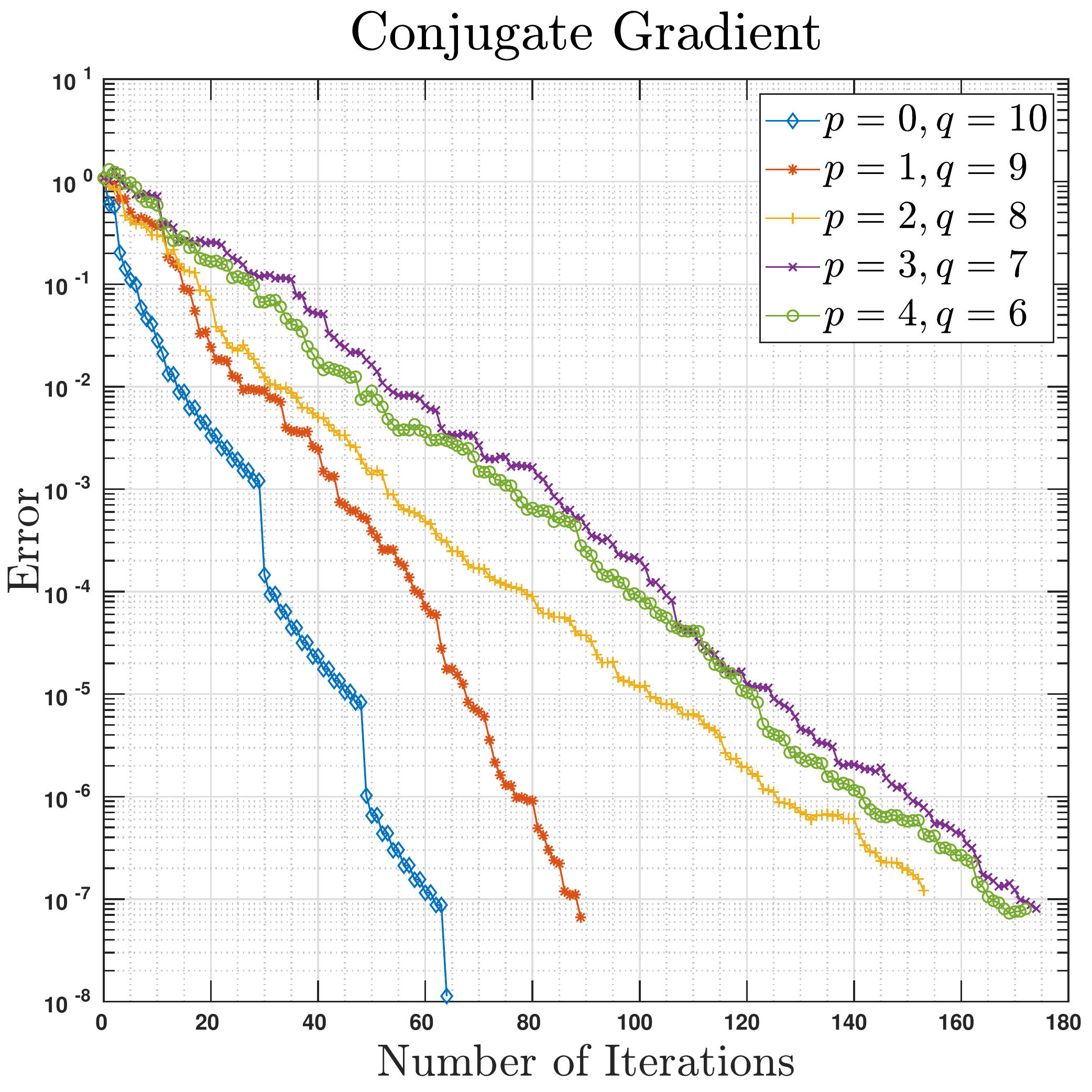}
  \caption{Semi-log convergence plots for Riemannian and semi-Riemannian steepest descent \cref{alg:gd-semi-mfld} and conjugate gradient \cref{alg:cg-semi-mfld} applied to a random instance of optimization problem \cref{eq:euc-sphere-minkowski} on the Euclidean unit sphere in Minkowski Space $\mathbb{R}^{p,q}$ with $p+q=10$. In this example, $A$ is a $10$-by-$10$ symmetric matrix (not necessarily positive definite), and the optimum is attained at the eigenvector associated with the largest eigenvalue of $A$; the vertical axes stand for the squared Euclidean distance between the iterate $x_k$ and the true optimum (obtained using Riemannian trust region method). In both subplots, each curve represents a different Minkowski space $\mathbb{R}^{p,q}$, i.e., the same base space $\mathbb{R}^{p+q}=\mathbb{R}^{10}$ but endowed with a different semi-Riemannian structure; the two blue curves corresponding to the case $p=0$ and $q=10$ are Riemannian steepest descent and conjugate gradient algorithms, which appear to require the least number of iterations for both steepest descent and conjugate gradient algorithms. These figures indicate that the convergence of semi-Riemannian optimization algorithms is guaranteed regardless of the specific semi-Riemannian structure imposed on the manifold, though the convergence rates vary.}
  \label{fig:sphere-sd-cg}
\end{figure}

\subsection{Pseudo-spheres in Minkowski Spaces}
\label{sec:pseudo-spher-mink}

Since the pseudo-spheres (\Cref{sec:pseudo-spheres}) and pseudo-hyperbolic spaces (\Cref{sec:pseudo-hyperb-spac}) differ from each other only by an anti-isometry \cite[Lemma 24]{ONeill1983}, we will only consider pseudo-spheres in this numerical experiment. Note that, given an arbitrary point $x\in \mathbb{S}^{p,q}$ and a tangent direction $\T_x\mathbb{S}^{p,q}$, it is generally difficult to calculate Riemannian geodesics on pseudo-spheres explicitly (except for some particular cases where e.g. \emph{Clairaut's relation} holds, see \cite[Chapter 3 Ex. 1]{doCarmo1992RG}), but semi-Riemannian geodesics adopt closed-form expression \cref{eq:semi-riem-geod-pseudo-sphere} and thus can be used as retractions for semi-Riemannian optimization algorithms. We consider the problem
\begin{equation}
  \label{eq:pseudo-sphere-minkowski}
  \min_{x\in\mathbb{S}^{p,q}}\left\| x-\xi \right\|_2^2
\end{equation}
where $\left\| x-\xi \right\|_2$ is the \emph{Euclidean} distance between $x\in\mathbb{S}^{p,q}$ and an arbitrarily chosen point $\xi\in \mathbb{R}^{p+q}$ that does not lie on $\mathbb{S}^{p,q}$. An illustration of the convergence of semi-Riemannian steepest descent and conjugate gradient methods (in semi-log scale) for a random instance of \cref{eq:pseudo-sphere-minkowski} with $p=3$ and $q=12$ can be found in \cref{fig:pseudosphere-sd-cg}, where the vertical axis marks the squared Euclidean distance between $x_k$ and the ground truth solution $x_{\textrm{true}}$ computed using the constrained optimization routine \verb|fmincon| provided in the \textsc{Matlab} optimization toolbox. This numerical experiment indicates that the convergence rates of both semi-Riemannian first-order methods are linear for \cref{eq:pseudo-sphere-minkowski}.

\begin{figure}[htbp]
  \centering
  \includegraphics[width=1.00\textwidth]{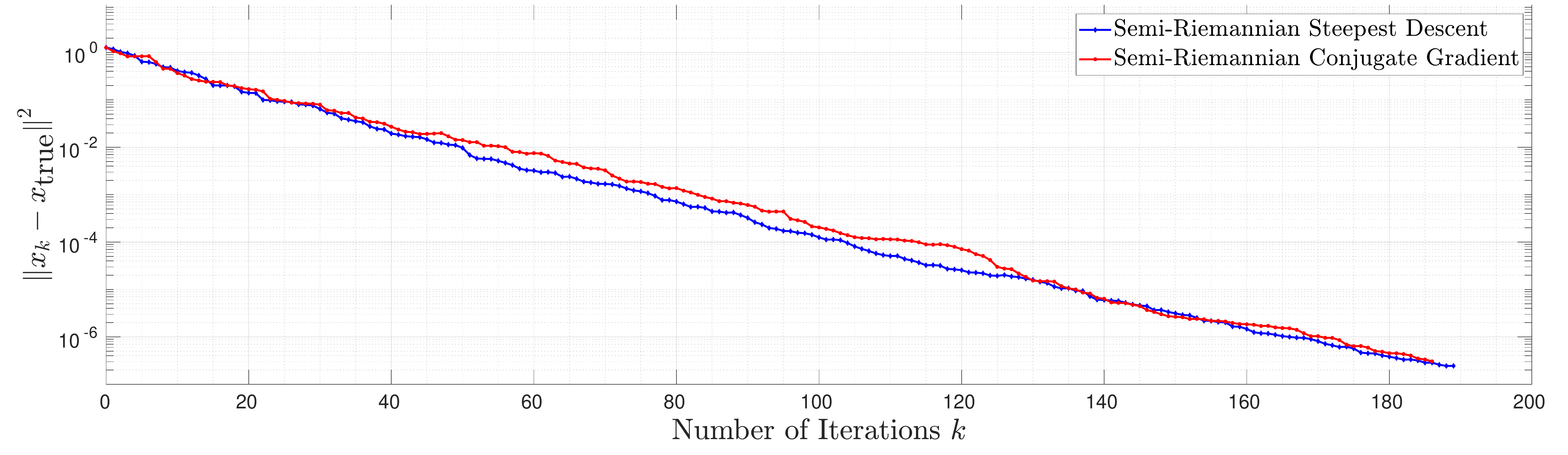}
  \caption{Semi-log convergence plot for semi-Riemannian steepest descent \cref{alg:gd-semi-mfld} and conjugate gradient \cref{alg:cg-semi-mfld} applied to a random instance of optimization problem \cref{eq:pseudo-sphere-minkowski} on the unit pseudo-sphere $\mathbb{S}^{p,q}$ in Minkowski Space $\mathbb{R}^{p,q}$ with $p=3$ and $q=12$. The convergence is measured with respect to the ground truth solution $x_{\mathrm{true}}$ computed directly using the constrained optimization functionality provided in the \textsc{Matlab} optimization toolbox. Linear convergence is demonstrated for both semi-Riemannian optimization algorithmsz.}
  \label{fig:pseudosphere-sd-cg}
\end{figure}

\section{Conclusion}
\label{sec:discussion}

Motivated by the metric independence of Riemannian optimization algorithms and the Riemannian geometry of self-concordant barrier functions, we developed an algorithmic framework for optimization on semi-Riemannian manifolds in this paper, which includes Riemannian manifold optimization and standard unconstrained optimization in Euclidean spaces as special cases. We proposed a modification to the semi-Riemannian gradients for obtaining descent directions, and used this methodology to devise steepest descent and conjugate gradient algorithms for semi-Riemannian manifold optimization. We also showed that second-order methods such as Newton's method and trust region methods are invariant with respect to difference choices of semi-Riemannian (including Riemannian) metrics. We provided numerical experiments to demonstrate the feasibility of the proposed algorithmic framework on non-degenerate semi-Riemannian submanifolds of Minkowski spaces. We defer more rigorous theoretical analysis, as well as broader ranges of applications of, semi-Riemannian manifold optimization to future work.

\section*{Software}
\verb|MATLAB| code for the surface registration algorithm is publicly available at \url{https://github.com/trgao10/SemiRiem}.

\section*{Acknowledgments}
The authors would like to thank Lin Lin (UC Berkeley) for inspirational discussions.

%\nocite{*}
\bibliographystyle{siamplain}
\bibliography{refs}

%\includepdf[pages=-]{supplement.pdf}

\appendix

\section{Genericity Non-degeneracy of Semi-Riemannian Structures on Hyperplanes of Minkowski Spaces}
\label{sec:gener-hyperpl-mink}

We begin with a brief discussion for the Gauss map defined in \Cref{sec:semi-riemannian-sub}. Consider $Z = \{(x,W)\in X\times \Gr(m,p+q):W = \N_x(X,\mathbb{R}^{p+q})\}$. Since $W = \N_x(X,\mathbb{R}^{p+q})$ if and only if $W$ is perpendicular to $\T_x X$ with respect to the Euclidean metric on $\mathbb{R}^{p+q}$, $Z$ is a closed subset of $X\times \Gr(m,p+q)$. More precisely, we have
\begin{proposition}
Let $\pi:X\times \Gr(m,p+q) \to \Gr(m,p+q)$ be the canonical projection onto the second factor. The following facts hold:
\begin{enumerate}[(1)]
  \item $\pi(Z) = \N(X)$;
  \item If $X$ is compact, then $\pi$ is a closed map, (i.e. mapping closed sets to closed sets). In particular, $\N(X)\subseteq \Gr(m,p+q)$ is a closed subset if $X$ is compact.
\end{enumerate}
\end{proposition}

The non-degeneracy of a generic hyperplane can be easily obtained as a corollary of \cref{prop:criterion-non-degeneracy}. Recall that hyperplanes $H\subset \mathbb{R}^{p+q}$ can be characterized as
\begin{equation*}
  H:=\left\{ x=\left( x_1,\cdots,x_p,y_1,\cdots,y_q \right)\in\mathbb{R}^{p+q}\,\Bigg|\,\sum_{j=1}^p a_j x_j + \sum_{j=1}^q b_j y_j = 0 \right\},
\end{equation*}
and we denote
\begin{equation*}
  \mathbf{n}\coloneqq (a_1,\dots,a_p,b_1,\dots,b_q)\in \mathbb{R}^{p+q}
\end{equation*}
for the \emph{normal vector} of $H$. Obviously, vector $\I_{p,q}\mathbf{n}$ lies in $H$ if and only if $\mathbf{n}$ satisfies the equation $\langle \mathbf{n}, \mathbf{n} \rangle = 0$, or equivalently, the equation $\sum_{j=1}^p a_j^2 = \sum_{j=1}^q b_j^2$.

\begin{corollary}
A generic hyperplane in $\mathbb{R}^{p,q}$ is non-degenerate. 
\end{corollary}
\begin{proof}
If $H$ is a hyperplane, then the image of its Gauss map is a single point $\mathbf{n}$, which is the line determined by the normal vector of $H$. Since $\mathcal{V}$ is a hypersurface in $\mathbb{P}\mathbb{R}^{p+q}$, we see that the normal vector of a generic hyperplane does not lie on $\mathcal{V}$ and hence a generic hyperplane is non-degenerate.
\end{proof}

\section{Additional Examples}
\label{sec:additional-examples}

The following \cref{table:summary} summarizes the examples computed in this paper.

\begin{table}[htpb]
\centering
\caption{Explicit Examples Calculated}
 \label{table:summary}
\begin{tabular}{|l|c|c|c|}
\hline
{\sc Manifolds}  & {Non-degenerate} & {\sc Geodesics} & {\sc Paral. Transp.} \\
\hline
(generic) hyperplane & yes & \checkmark & \checkmark \\
\hline
sphere & no & \ding{55} & \ding{55} \\
\hline
pseudo-sphere & yes & \checkmark  & \checkmark \\
\hline
pseudo-hyperbolic space & yes & \checkmark  & \checkmark \\
\hline
indefinite orthogonal group & no & \checkmark & \ding{55} \\
\hline
orthogonal group & no & \checkmark & \ding{55}  \\
\hline
special linear group & yes ($p\ne q$) & \checkmark & \ding{55} \\
\hline
symplectic group & no & \checkmark & \ding{55} \\
\hline 
SPD matrices  & no & \checkmark & \ding{55}\\
\hline
\end{tabular}
\end{table}

\subsection{Semi-Riemannian Geometry of Symmetric Positive Definite Matrices}
\label{sec:semi-riem-geom-additional-example}

Let $\mathbb{S}^n_{++}$ be the manifold consisting of all $n\times n$ symmetric positive definite matrices. A matrix $A\in \mathbb{R}^{n\times n}$ is symmetric positive definite if and only if there exists some $M\in \GL(n,\mathbb{R})$ such that $A = M^\mathsf{T}M$. The tangent space of $\mathbb{S}^n_{++}$ at $A$ is $\operatorname{S}^2 \mathbb{R}^n$. Hence if we regard $\mathbb{S}^n_{++}$ as a semi-Riemannian sub-manifold of $\GL(n,\mathbb{R})$ with respect to the semi-Riemannian metric $\langle \cdot, \cdot \rangle$ with signature $(p,q) = (p,n-p)$, then the semi-normal space of $\mathbb{S}^n_{++}$ at $A$ is 
\[
\SN_A (\mathbb{S}^n_{++}, \GL(n,\mathbb{R})) = \left\lbrace
\I_{p,q}\Delta:\Delta\in \bigwedge^2 \mathbb{R}^2
 \right\rbrace.
\]
It is straightforward to compute the intersection of $\T_A \mathbb{S}^n_{++}$ and $\SN_A (\mathbb{S}^n_{++}, \GL(n,\mathbb{R}))$. Hence we obtain the following:
\begin{proposition}
The degenerate bundle of $\mathbb{S}^n_{++}$ in $\GL(n,\mathbb{R})$ is 
\[
(\mathbb{S}^n_{++})_A^{\perp} = \left\lbrace 
\begin{bmatrix}
0 & X^\mathsf{T} \\
X & 0
\end{bmatrix}: X\in \mathbb{R}^{q\times p}
\right\rbrace.
\]
In particular, $\mathbb{S}^n_{++}$ is a degenerate semi-Riemannian sub-manifold of $\GL(n,\mathbb{R})$.
\end{proposition}

\begin{proposition}
If a geodesic passing through $A\in  \mathbb{S}^n_{++}$ with the tangent direction $\I_{p,q} \Delta$ exists, then $\Delta = \begin{bmatrix}
0 & -\Delta_2^\mathsf{T} \\
\Delta_2 & 0
\end{bmatrix}$ and $\gamma(t)$ can be written as 
\[
\gamma(t) = \int_0^t  \begin{bmatrix}
0 & U(\tau)^\mathsf{T} \\
U(\tau) & 0
\end{bmatrix} + A,
\]
for some suitable $q\times p$ matrix-valued function $U(t)$ such that $U(0) = \Delta_2$ and \[
t \begin{bmatrix}
0 & U(\tau)^\mathsf{T} \\
U(\tau) & 0
\end{bmatrix} d\tau + A \succ 0
\]
for each $t$.
\end{proposition}
\begin{proof}
Let $\gamma(t)$ be a geodesic curve passing through $A$ with the tangent direction $\I_{p,q}\Delta$. Then $\gamma(t)$ must satisfies the following conditions:
\[
\gamma(0) = A, \quad \dot{\gamma}(0) =\I_{p,q} \Delta,\quad \dot{\gamma}(t)\in \operatorname{S}^2 \mathbb{R}^n, \quad \I_{p,q}\ddot{\gamma}(t) \in \bigwedge^2 \mathbb{R}^2.
\]
Hence $\dot{\gamma}(t) = \begin{bmatrix}
0 & U(t)^\mathsf{T} \\
U(t) & 0
\end{bmatrix}$ for some $q\times p$ matrix-valued function $U(t)$. Therefore, we must have 
\[
\gamma(t) = \int_0^t  \begin{bmatrix}
0 & U(t)^\mathsf{T} \\
U(t) & 0
\end{bmatrix} + A\in   \mathbb{S}^n_{++}.
\]
Since $A = \frac{A}{t}\int_0^t d\tau$, we see that $\gamma(t) \succ 0$ if and only if 
\[
t \begin{bmatrix}
0 & U(t)^\mathsf{T} \\
U(t) & 0
\end{bmatrix} + A \succ 0.
\]
\end{proof}

%\subsection{Other Matrix Lie Groups}
%\label{sec:other-matrix-lie}

\subsection{Semi-Riemannian Geometry of Matrix Lie Groups}
\label{sec:semi-riem-geom-matrix-lie-groups}

As discussed in \Cref{sec:degen-subm}, matrix Lie groups provide another class of rich examples for manifolds admitting semi-Riemannian structures. In this section we illustrate how to apply the semi-Riemannian manifold optimization framework developed in \Cref{sec:semi-riem-optim} to some common matrix Lie groups, despite the degeneracy of their inherited sub-Riemannian structures, by means of projection to semi-normal bundles. We begin with a general discussion on the semi-normal bundle of matrix Lie groups, then specialize to several examples. %More examples can be found in \Cref{sec:other-matrix-lie}.

Let $G\subseteq \GL(n,\mathbb{R})$ be a matrix Lie group. We denote by $\mathfrak{g}$ the Lie algebra of $G$. Then the tangent space of $G$ at a point $A\in G$ is simply $A\mathfrak{g}$. For each fixed positive integer $0 \le p \le n$, there is a left-invariant semi-Riemannian metric of signature $(p,q) = (p,n-p)$ on $\GL(n,\mathbb{R})$ defined by
\[
\langle U,V \rangle_A = \tr((A^{-1}U)^\mathsf{T}\I_{p,q}(A^{-1}V)),\quad U,V\in \T_A G\quad, A\in G.
\]
Since the bilinear form $\langle\cdot, \cdot \rangle$ is left-invariant, we have the following 
\begin{lemma}\label{lemma:semi-normal left invariant}
For each $A\in G$, we have 
\[
\SN_A (G,\GL(n,\mathbb{R})) = \{AX: X\in \SN_{\I_n} (G,\GL(n,\mathbb{R})\}.
\]
\end{lemma}
Let $(\cdot,\cdot)$ be the Riemannian metric on $\GL(n,\mathbb{R})$ defined by 
\[
(U,V)_A = \tr((A^{-1}U)^\mathsf{T} A^{-1}V),\quad A\in \GL(n,\mathbb{R}),\quad U,V\in \T_A\GL(n,\mathbb{R}).
\]
We denote by $\N(G,\GL(n,\mathbb{R}))$ the normal bundle of $G$ in $\GL(n,\mathbb{R})$. We can relate the normal bundle and the semi-normal of bundle of $G$ in $\GL(n,\mathbb{R})$ by the following:
\begin{proposition}\label{prop:normal v.s. semi-normal}
For each $A\in G$, we have 
\[
\SN_A(G,\GL(n,\mathbb{R})) = \I_{p,q} \N_A(G,\GL(n,\mathbb{R})).
\]
In particular, $\SN(G,\GL(n,\mathbb{R}))$ is a vector bundle on $G$ of rank $n^2 - \dim G$. 
\end{proposition}
\begin{proof}
$U\in \mathbb{R}^{n\times n}$ lies in $\SN_A(G,\GL(n,\mathbb{R}))$ if and only if $\tr((A^{-1}U)^\mathsf{T} \I_{p,q} A^{-1}V) = 0$ for all $V\in \mathfrak{g}$. This implies that $U\in \SN_A(G,\GL(n,\mathbb{R}))$ if and only if $\I_{p,q}U\in \N_A(G,\GL(n,\mathbb{R}))$ and this completes the proof.
\end{proof}

Let $\gamma(t)$ be a geodesic passing through $A\in G$ with tangent direction $AU$. Then by definition, $\gamma(t)$ is characterized by the following relations:
\[
\gamma(t)\in G,\quad \dot{\gamma}(t)\in \gamma(t)\mathfrak{g},\quad \ddot{\gamma}(t)\in \SN_{\gamma(t)}(G,\GL(n,\mathbb{R}))
\]
and the initial condition $\gamma(0) = A$, $\dot{\gamma}(0) = U$. To compute geodesics explicitly for matrix Lie groups, we need the following simple but handy observations. 
\begin{lemma}\label{lemma:U(t)}
If $\gamma(t)$ is a given geodesic, then there exists a unique curve $U(t)$ in $\mathfrak{g}$ such that $\dot{\gamma}(t) = \gamma(t)U(t)$ and $U(t)^2 + \dot{U(t)} \in \SN_{\I_n}(G,\GL(n,\mathbb{R}))$.
\end{lemma}
\begin{proof}
Since $\dot{\gamma}(t)\in \gamma(t)\mathfrak{g}$, we can write $\dot{\gamma}(t) = \gamma(t)U(t)$ for a curve $U(t)$ in $\mathfrak{g}$. By differentiating $\dot{\gamma}(t) = \gamma(t)U(t)$, we obtain 
\[
\ddot{\gamma}(t) = \gamma(t)(U(t)^2 + \dot{U}(t)) \in \SN_{\gamma(t)}(G,\GL(n,\mathbb{R})).
\]
This implies that $U(t)^2 + \dot{U(t)} \in \SN_{\I_n}(G,\GL(n,\mathbb{R}))$.
\end{proof}

\begin{proposition}\label{prop:U(t)}
If $\gamma(t)$ is a geodesic passing through $A$ with tangent direction $AU$, then $\gamma(t) = A \exp(\int_{0}^t U(\tau) d\tau)$, where $U(t)$ is the curve in $\mathfrak{g}$ determined in \cref{lemma:U(t)}. 
\end{proposition}

\begin{corollary}\label{cor:geodesic-constant speed}
The geodesic curve $\gamma(t)$ passing through $A\in \O(p,q)$ with direction $A\I_{p,q} U$ in \cref{eqn: characterization of geodesics 2} is of constant speed.
\end{corollary}
\begin{proof}
We calculate $\widetilde{g}_{\gamma(t)}(\dot{\gamma}(t),\dot{\gamma}(t))$. By \cref{lemma: characterization of geodesics}, we have 
\[
\widetilde{g}_{\gamma(t)}(\dot{\gamma}(t),\dot{\gamma}(t))  =- \tr(U_1^2) +  \tr(U_3^2) . 
\]
\end{proof}

\begin{corollary}\label{cor:geodesic-energy}
The energy $\E(t)$ of the geodesic $\gamma(t)$ in \cref{eqn: characterization of geodesics 2} is 
\begin{equation*}
  E \left( t \right)=(- \tr(U_1^2) +  \tr(U_3^2))t.
\end{equation*}
\end{corollary}
\begin{proof}
By definition and \cref{cor:geodesic-constant speed}, we have
\[
\E(t) = \int_0^t \widetilde{g}_{\gamma(\tau)} (\dot{\gamma}(\tau),\dot{\gamma}(\tau)) d\tau =  (- \tr(U_1^2) +  \tr(U_3^2)) t.
\]
\end{proof}

The following characterization of the parallel transport of a tangent vector along a geodesic on $G$ can be easily obtained by unravelling \cref{defn:parallel transport}.
\begin{lemma}
Let $\Delta(t)$ be a parallel transport of $\Delta\in \T_A G$ along a geodesic curve $\gamma(t)$ passing through $A\in G$ with  tangent direction $AU \in \T_A G$. Then
\[
\Delta(t) = \gamma(t) V(t),
\]
where $V(t)$ is a curve in $\mathfrak{g}$ such that $U(t)V(t) + \dot{V}(t)\in \SN_{\I_n}(G,\GL(n,\mathbb{R}))$ and $V(0) = \Delta$. Here $U(t)$ is the curve in $\mathfrak{g}$ determined in \cref{lemma:U(t)}.
\end{lemma}

\subsubsection{Indefinite Orthogonal Groups}
\label{sec:indef-orth-groups}

The definition of indefinite orthogonal groups $O \left( p,q \right)$ can be found in \cref{exm:indef-orth-group}. In this subsection we derive explicit formulae for the geodesic in $O \left( p,q \right)$ following \cref{defn:geodesic}.

\begin{proposition}\label{prop:SN-indef}
For each $A\in \O(p,q)$, we have 
\[
\SN_A(\O(p,q),\GL(p+q,\mathbb{R})) = \{ AS: S\in \operatorname{S}^2 \mathbb{R}^{p+q} \}
\]
\end{proposition}
\begin{proof}
By \cref{lemma:semi-normal left invariant}, it is sufficient to prove
\[
\SN_{\I_{p+q}}(\O(p,q),\GL(p+q,\mathbb{R})) = \{ S: S\in \operatorname{S}^2 \mathbb{R}^{p+q} \}
\]
as $\langle \cdot  , \cdot \rangle$ is left-invariant. To this end, we notice that by \cref{prop:normal v.s. semi-normal}
\[
\SN_{\I_n}(\O(p,q),\GL(p+q,\mathbb{R})) = \I_{p,q} \N_{\I_n}(\O(p,q),\GL(n,\mathbb{R})).
\]
Now $\N_{\I_n}(\O(p,q),\GL(n,\mathbb{R}))$ consists of all matrices of the form $\I_{p,q}S$ where $S$ is symmetric, we may conclude that 
\[
\SN_{\I_n}(\O(p,q),\GL(p+q,\mathbb{R}))  = \{S: S\in \operatorname{S}^2 \mathbb{R}^{p+q}\}.
\]
\end{proof}

Let $\gamma(t)$ be a geodesic curve passing through $A\in \O(p,q)$ with direction $A\I_{p,q}\Delta$ for some skew-symmetric matrix $\Delta$. By \cref{prop:U(t)} and \cref{prop:SN-indef}, the curve $\gamma(t)$ can be written as 
\begin{equation}\label{eqn:geodesic for O(p,q)}
    \gamma(t) = A \exp\left(\int_0^t U(\tau) d\tau\right),
\end{equation}
where $U(t)$ is a curve in $\mathfrak{o}(p,q)$ such that 
\begin{equation}\label{eqn:U(t) for O(p,q)}
    U(0) = \I_{p,q} \Delta,\quad U(t)^2 + \dot{U(t)} \in \operatorname{S}^2 \mathbb{R}^{p+q}.
\end{equation}
We partition a $(p+q)\times (p+q)$ skew-symmetric matrix $Y$ as $\begin{bmatrix}
Y_1 & -Y_2^\mathsf{T} \\
Y_2 & Y_3
\end{bmatrix}$ where $Y_1\in \bigwedge^2 \mathbb{R}^p$, $Y_3\in \bigwedge^2 \mathbb{R}^q$ and $Y_2\in \mathbb{R}^{q\times p}$.

\begin{lemma}\label{lemma: characterization of geodesics}
Let $t\mapsto \gamma(t)$ be the geodesic passing through $A$ with direction $A\I_{p,q} \Delta$. Then the curve $U(t)$ in $\mathfrak{o}(p,q)$ satisfying \eqref{eqn:U(t) for O(p,q)} is  
\begin{equation}\label{eqn: characterization of geodesics 1}
U(t) = \begin{bmatrix}
-\Delta_1 & \exp(-\Delta_1 t) \Delta_2^\mathsf{T} \exp(\Delta_3 t) \\
\exp (-\Delta_3 t) \Delta_2 \exp (\Delta_1t) & \Delta_3
\end{bmatrix}.
\end{equation}
\end{lemma}
\begin{proof}
We parametrize $U(t)$ as $U(t) = \I_{p,q} \Delta(t)$ where $\Delta(t)\in \bigwedge^2 \mathbb{R}^{p+q}$. Since $U(t)^2 + \dot{U}(t)$ is symmetric, we have 
\[
\operatorname{Skew}(U(t)^2 + \dot{U}(t)) = 0.
\]
This implies that
\[
\dot{\Delta_1}(t) =0,\quad \dot{\Delta_3}(t) = 0, \quad \dot{\Delta_2}(t) =  \Delta_2(t) \Delta_1(t) - \Delta_3(t) \Delta_2(t).
\]
from which we obtain $\Delta_1(t) = \Delta_1$, $\Delta_3(t) = \Delta_3$ and
\[
\Delta_2(t) = \exp(-\Delta_3 t) \Delta_2 \exp(\Delta_1 t).
\]
Therefore, we obtain 
\[
U(t) = \begin{bmatrix}
-\Delta_1 & \exp(-\Delta_1 t) \Delta_2^\mathsf{T} \exp(\Delta_3 t) \\
\exp (-\Delta_3 t) \Delta_2 \exp (\Delta_1t) & \Delta_3
\end{bmatrix}.
\]
\end{proof}
By \cref{prop:U(t)}, we obtain the following:
\begin{proposition}\label{prop:characterization-of-geodesics-indef}
The geodesic curve $\gamma(t)$ passing through $A\in \O(p,q)$ with direction $A\I_{p,q}\Delta$ is unique and is 
\begin{equation}\label{eqn: characterization of geodesics 2}
\gamma(t) =A \exp \left( \begin{bmatrix}
-\Delta_1 t &  \int_0^t \exp(-\Delta_1\tau) \Delta_2^\mathsf{T} \exp(\Delta_3\tau) d\tau  \\
 \int_0^t \exp(-\Delta_3\tau) \Delta_2 \exp(\Delta_1\tau) d \tau & \Delta_3 t
\end{bmatrix}  \right).
\end{equation} 
\end{proposition}

\begin{corollary}\label{cor:geodesic-exponential}
The curve $\gamma(t) = A \exp(t \I_{p,q} \Delta)$ is a geodesic curve passing through $A$ with direction $\I_{p,q}Y$ if and only if $\Delta_3\Delta_2 = \Delta_2\Delta_1 $.
\end{corollary}
\begin{proof}
If $\gamma(t) = A \exp(t\I_{p,q} \Delta)$ is a geodesic curve, then from \cref{prop:characterization-of-geodesics-indef}, it is straightforward to verify that $\Delta_3\Delta_2 - \Delta_2\Delta_1 = 0$. Conversely, if $\Delta_3\Delta_2 - \Delta_2\Delta_1 = 0$, then \cref{eqn: characterization of geodesics 2} is reduced to 
\[
\gamma(t) = A \exp \left( \begin{bmatrix}
\Delta_1 t & \Delta_2^\mathsf{T} t \\
\Delta_2 t & \Delta_3 t
\end{bmatrix} \right) = A \exp (t \I_{p,q} \Delta).
\]
\end{proof}
In particular, if $p=0$ (resp. $q=0$), then $\Delta_3\Delta_2 = \Delta_2 \Delta_1$ always holds and \cref{cor:geodesic-exponential} shows that geodesics in $\O(q)$ (resp. $p=0$) are of the form $A \exp(t \Delta)$ for some skew-symmetric matrix $\Delta$. Moreover, \cref{cor:geodesic-exponential} already shows a big difference between Riemannian geometry and non-Riemannian geometry of $\O(p,q)$ (cf. \cite[Theorem 2.14]{Andruchow2014}).

\subsubsection{Orthogonal Groups}
\label{sec:orth-groups}

By \cref{prop:characterization-of-geodesics-indef}, the geodesic passing through a point $A\in \O(p,q)$ with tangent direction $U\in \T_A \O(p,q)$ exists. We will see that this is not always true. To that end, we consider $\O(n)$ for example. First, according to \cref{lemma:semi-normal left invariant} and \cref{prop:normal v.s. semi-normal}, we have the following
 \begin{proposition}
 For each $A\in \O(n)$, the semi-normal space of $\O(n)$ at $A$ is 
 \[
 \SN_A(\O(n),\GL(n,\mathbb{R})) = \{\I_{p,q}S: S\in \operatorname{S}^2\mathbb{R}^n \}.
 \]
 \end{proposition}
 Now if $\gamma(t)$ is a geodesic curve passing through $A$ with tangent direction $A\Delta$ where $\Delta = \begin{bmatrix}
 \Delta_1 & -\Delta_2^\mathsf{T} \\
 \Delta_2 & \Delta_3
 \end{bmatrix}\in \mathfrak{o}(n)$, then 
 \[
 \gamma(t) = A \exp \left(\int_0^t U(\tau) d\tau\right),
 \]
 where $U(t) = \begin{bmatrix}
 U_1(t) & -U_2(t)^\mathsf{T} \\
 U_2(t) & U_3(t)
 \end{bmatrix}$ is skew-symmetric satisfying 
 \[
 U(0) = \Delta, \quad \I_{p,q} (U(t)^2 + \dot{U}(t)) \in \operatorname{S}^2 \mathbb{R}^n.
 \]
 Hence we have 
 \[
 \operatorname{Skew}(\I_{p,q} (U(t)^2 + \dot{U}(t))) = 0,
 \]
 which implies  
 \[
 \dot{U_1}(t) = 0,\quad \dot{U_3}(t) =0,\quad U_2(t) U_1(t) + U_3(t) U_2(t) =0.
 \]
 Therefore, we may conclude that 
 \[
 U_1(t) = \Delta_1,\quad U_3(t) = \Delta_3,\quad U_2(t)\Delta_1 + \Delta_3 U_2(t) = 0.
 \]
 \begin{proposition}\label{prop:geodesic O(n)}
 On $\O(n)$, a geodesic passing through $A$ with the tangent direction $A\Delta$ exists if and only if the skew symmetric matrix $\Delta = \begin{bmatrix}
 \Delta_1 & -\Delta_2^\mathsf{T} \\
 \Delta_2 & \Delta_3
 \end{bmatrix}$ satisfies the relation 
 \[
 \Delta_2 \Delta_1 + \Delta_3 \Delta_2 = 0.
 \]
 If such a $\gamma$ exists, then $\gamma$ if of the form 
 \[
 \gamma(t) = A \exp \left( \begin{bmatrix}
 \Delta_1 t & -\int_0^t U_2(\tau)^\mathsf{T} d\tau \\
\int_0^t U_2(\tau) d\tau & \Delta_3t
\end{bmatrix}\right) 
\]
where $U_2$ is a curve in the linear subspace $\{X\in \mathbb{R}^{q\times p}: X \Delta_1 + \Delta_3 X = 0\}$ such that $U_2(0) = \Delta_2$. The square of the speed of $\gamma(t)$ is $\tr(\Delta^2_1) - \tr(\Delta^2_3)$ and the energy of $\gamma$ is $E(t) = t(\tr(\Delta^2_1) - \tr(\Delta^2_3))$.
\end{proposition}
In particular, if $\Delta_2 \Delta_1 + \Delta_3 \Delta_2 =0$ and we take $U_2(t) = \Delta_2$, then by Proposition \ref{prop:geodesic O(n)} $\gamma$ becomes 
\[
 \gamma(t) = A \exp(\Delta t),
 \]
 which is exactly the geodesic curve on $\O(n)$ with the usual Riemannian metric
 \begin{equation*}
   \left(U, V\right)_A = \tr((A^{-1}U)^\mathsf{T} A^{-1}V)
 \end{equation*}
where $A\in \O(n)$ and $U,V\in \T_A\O(n)$.

\subsubsection{Special Linear Groups with $p\neq q$}
\label{sec:spec-linear-group}

We notice that the Lie algebra of $\SL(n,\mathbb{R})$ consists of all traceless $n\times n$ matrices, which implies that
\[
\N_{\I_n} (\SL(n,\mathbb{R}),\GL(n,\mathbb{R})) = \{\lambda \I_{n}:\lambda\in \mathbb{R}\}.
\]
Therefore, from \cref{lemma:semi-normal left invariant} and \cref{prop:normal v.s. semi-normal}, we have the following
 \begin{proposition}\label{prop:semi-normal SL}
  For each $A\in \SL(n,\mathbb{R})$, the semi-normal space of $\SL(n,\mathbb{R})$ at $A$ is 
\[
\SN_A(\SL(n,\mathbb{R},\GL(n,\mathbb{R}))) = \{\lambda A\I_{p,q}: \lambda\in \mathbb{R}\}.
\]
In particular, 
\[
\SL(n,\mathbb{R})^\perp = \begin{cases}
\SL(n,\mathbb{R})\times \{0\},&\text{if}~p\ne q\\
\SN(\SL(n,\mathbb{R},\GL(n,\mathbb{R}))),&\text{otherwise}.
\end{cases}
\]
Hence $\SL(n,\mathbb{R})$ is a non-dengerate semi-Riemannian sub-manifold of $\GL(n,\mathbb{R})$ if and only if $p\ne q$.
\end{proposition}

By \cref{prop:semi-normal SL} and \cite[Corollary 10]{ONeill1983}, we obtain the following
\begin{corollary}
If $p\ne q$, then the geodesic passing through $A\in \SL(n,\mathbb{R})$ with the tangent direction $AU\in \T_A \SL(n,\mathbb{R})$ exists and it is unique.
\end{corollary}

If $\gamma(t)$ is a geodesic curve passing $A\in \SL(n,\mathbb{R})$ with the tangent direction $A U\in \T_A \SL(n,\mathbb{R})$, then 
 \[
\gamma(t) = A \exp\left(\int_0^t U(\tau) d\tau\right),
\]
where $\tr(U(t)) = 0$ satisfying 
\begin{equation}\label{eqn:U(t) for SL}
    U(0) = U, \quad  U(t)^2 + \dot{U}(t) = \lambda(t) \I_{p,q}
\end{equation}
for some real valued function $\lambda(t)$. Moreover, from \eqref{eqn:U(t) for SL}, we also have 
\[
U(0)^2 + \dot{U}(0) = \lambda(0) \I_{p,q}.
\]
Together with the fact that $\tr(\dot{U}(0)) = 0$ as $\tr(U(t)) = 0$, we may conclude that 
\begin{equation}\label{eqn:U(t) for SL-1}
\lambda(0)(q-p) = \tr(U(0)^2) = \tr(U^2).
\end{equation}
\begin{proposition}
If $p=q$ and $\gamma$ is a geodesic curve passing through $A\in \SL(n,\mathbb{R})$ with the tangent direction $\dot{\gamma}(0) = AU\in \T_A \SL(n,\mathbb{R})$, then $\tr(U^2) = 0$.
\end{proposition}

Next we consider the simplest case where $n=2$ and $p=q =1$, we may write $U(t) = \begin{bmatrix}
a(t) & b(t) \\
c(t) & -a(t)
\end{bmatrix}$ and $U = \begin{bmatrix}
a & b \\
c & -a
\end{bmatrix}$ so that the ODE in \eqref{eqn:U(t) for SL} becomes 
\[
  \begin{bmatrix}
a(t)^2  + b(t)c(t) + \dot{a}(t) & \dot{b}(t) \\
\dot{c}(t) & a(t)^2 + b(t)c(t) - \dot{a}(t)
\end{bmatrix} = \lambda(t) \begin{bmatrix}
-1 & 0 \\
0 & 1
\end{bmatrix}.
\]
This implies that $b(t) = b, c(t) = c$ and $a(t)^2 = -bc$. Hence we may conclude that 
 \begin{proposition}
An embedded geodesic $\gamma$ on $\SL(2,\mathbb{R})$ passing through $A$ with tangent direction $U = \begin{bmatrix}
a & b \\
c & -a
\end{bmatrix}$ exists if and only if $a^2 + bc = 0$. Moreover, if such $\gamma$ exists, then it is unique and 
\[
\gamma(t) = A \exp\left(\begin{bmatrix}
a & b \\
c & -a
\end{bmatrix}t\right). 
\]
The square of the speed of $\gamma(t)$ is $2a^2 + b^2 + c^2$ and the energy $E(t)$ is $t(2a^2 + b^2 + c^2)$.
\end{proposition}

\subsubsection{Symplectic Group} We recall that the symplectic group $\Sp(2n,\mathbb{R})$ is the group of $(2n) \times (2n)$ matrices $A$ satisfying
\[
A^\mathsf{T} \J_n A = \J_n
\]
where $\J_n = \begin{bmatrix}
0 & \I_n \\
-\I_n & 0
\end{bmatrix}$. The Lie algebra of $\Sp(2n,\mathbb{R})$ is 
\[
\mathfrak{sp}(n) = \left\lbrace \J_nS: S\in \operatorname{S}^2\mathbb{R}^{2n} \right\rbrace
\]
and the normal space is
\[
\N_{\I_{2n}}(\Sp(2n,\mathbb{R}),\GL(2n,\mathbb{R})) =  \left\lbrace \J_n \Delta: \Delta \in \bigwedge^2 \mathbb{R}^{2n} \right\rbrace.
\]
Without loss of generality, we may assume that $1\le p \le n$. 
\begin{proposition}
For each $A\in \Sp(2n,\mathbb{R})$, we have 
\[
\SN_A(\Sp(2n,\mathbb{R}),\GL(2n,\mathbb{R})) = \left\lbrace A \I_{p,q} \J_n \Delta: \Delta \in \bigwedge^2 \mathbb{R}^{2n} \right\rbrace
\]
and 
\[
\Sp(2n,\mathbb{R})_{A}^\perp = \left\lbrace 
A\begin{bmatrix}
X & Y & 0 & Z^\mathsf{T} \\
0 & 0 & Z &   0\\
0 & 0 & X^\mathsf{T} & 0 \\
0 & 0 &  Y^\mathsf{T} & 0 
\end{bmatrix}: X\in \mathbb{R}^{p\times p},Y\in \mathbb{R}^{p \times (n-p)},Z\in \mathbb{R}^{(n-p) \times p}
\right\rbrace.
\]
In particular, $\Sp(2n,\mathbb{R})$ is a degenerate semi-Riemannian sub-manifold of $\GL(2n,\mathbb{R})$.
\end{proposition}
\begin{proof}
The description of the semi-normal space follows from Lemma \ref{lemma:semi-normal left invariant} and Proposition \ref{prop:normal v.s. semi-normal} and the description of $\Sp(2n,\mathbb{R})_A^\perp$ is obtained by a straightforward calculation.
\end{proof}

Next we describe geodesics on $\Sp(2n,\mathbb{R})$.
\begin{proposition}
Let $\gamma(t)$ be a geodesic passing through $A\in \Sp(2n,\mathbb{R})$ with the tangent direction $A\J_n X$. If we write $X$ as $X = \begin{bmatrix}
S & P^\mathsf{T} \\
P & Q
\end{bmatrix}$ where $S,Q$ are symmetric $n\times n$ matrices and $P$ is an $n\times n$ matrix. We also partition $S,P,Q$ as 
\[
S = \begin{bmatrix}
S_{11} & S_{12} \\
S_{12}^\mathsf{T} & S_{22}
\end{bmatrix},\quad
P = \begin{bmatrix}
P_{11} & P_{12} \\
P_{21} & P_{22}
\end{bmatrix},\quad Q = \begin{bmatrix}
Q_{11} & Q_{12} \\
Q_{12}^\mathsf{T} & Q_{22}
\end{bmatrix}
\]
where $S_{11},P_{11},Q_{11}\in \mathbb{R}^{p\times p}$ and $S_{22},P_{22},Q_{22}\in \mathbb{R}^{(n-p) \times (n-p)}$. Then 
\[
\gamma(t) =A \exp(\int_0^t \begin{bmatrix}
P(\tau) & Q(\tau) \\
- S & -P(\tau)^\mathsf{T}
\end{bmatrix}d \tau,
\]
where $P(t) = \begin{bmatrix}
P_{11}(t) & P_{12}(t) \\
P_{21} & P_{22}
\end{bmatrix}$, $Q(t) = \begin{bmatrix}
Q_{11} & Q_{12}(t) \\
Q_{12}(t)^\mathsf{T} & Q_{22}
\end{bmatrix}$ and 
\[
\begin{cases}
Q_{11}S_{11} + Q_{12}(t)S_{12}^\mathsf{T} = P_{11}(t)^2 + P_{12}(t) P_{21}, \\
Q_{11}S_{12} + Q_{12}(t) S_{22} = P_{11}(t) P_{12}(t) + P_{12}(t) P_{22},\\
(P_{21}Q_{11} + P_{22}P_{12}(t)^\mathsf{T}) - (Q_{12}(t)^\mathsf{T}P_{11}(t)^\mathsf{T} + Q_{22}P_{12}(t)^\mathsf{T}) + \dot{Q}_{12}(t)^\mathsf{T}=0.
\end{cases}
\]
\end{proposition}

\end{document}